\documentclass[a4paper,11pt]{article}
\usepackage{amsthm, amsmath, amssymb, paralist, calc, verbatim, tocloft, mathtools}
\usepackage{graphicx, pstricks}

\newcommand{\llangle}{\langle\!\langle}
\newcommand{\rrangle}{\rangle\!\rangle}

\newtheorem{theorem}{Theorem}[section]

\newtheorem{lemma}[theorem]{Lemma}

\newtheorem{corollary}[theorem]{Corollary}

{\theoremstyle{definition}
\newtheorem{remark}[theorem]{Remark}
\newtheorem{definition}[theorem]{Definition}

\newtheorem{construction}[theorem]{Construction}
\newtheorem{example}[theorem]{Example}
\newtheorem{examples}[theorem]{Examples}
}

\numberwithin{equation}{section}

\newcommand{\Z}{\mathbb{Z}}   % the integers
\newcommand{\A}{\mathcal{A}}%struct alg
\newcommand{\Ss}{\mathcal{S}}
\newcommand{\I}{\mathcal{I}} 

\newcommand{\XtensorE}{X\otimes_K \Delta}
\newcommand{\ga}{\gamma}
\newcommand{\la}{\lambda}
\newcommand{\si}{\sigma}

\newcommand{\FTS}{Freudenthal triple system}
\newcommand{\step}[2]{\medskip\noindent{\em Step} {#1}: \ {\em #2}.\medskip}

\DeclareMathOperator{\CD}{CD}
\DeclareMathOperator{\Char}{char}
\DeclareMathOperator{\Gal}{Gal}
\DeclareMathOperator{\End}{End}

\DeclareMathOperator{\Fix}{Fix}

\DeclareMathOperator{\Nrd}{Nrd}
\DeclareMathOperator{\Trd}{Trd}

\begin{document}

\title{Exceptional Moufang quadrangles and structurable algebras}
\author{Lien Boelaert%
%\and 
\quad
Tom De Medts%
}
\date{\today}
\maketitle

\begin{abstract}
	In 2000, J. Tits and R. Weiss classified all Moufang spherical buildings of rank two, also known as Moufang polygons.
	The hardest case in the classification consists of the Moufang quadrangles.
	They fall into different families, each of which can be described by an appropriate algebraic structure.
	For the exceptional quadrangles, this description is intricate and involves many different
	maps that are defined {\em ad hoc} and lack a proper explanation.

	In this paper, we relate these algebraic structures to two other classes of algebraic structures that had
	already been studied before, namely to Freudenthal triple systems and to structurable algebras.
	We show that these structures give new insight in the understanding of the corresponding Moufang quadrangles.
\end{abstract}

\vspace*{2ex} \noindent
{\tt MSC-2010} : 17A30, 17A40, 17C40, 20G15, 20G41

\noindent
{\tt keywords} : Moufang polygons, Moufang quadrangles, Freudenthal triple systems, quadrangular algebras, structurable algebras,
linear algebraic groups, exceptional groups

%%%%%%%%%%%%%%%%%%%%%%%%%%%%%%%%%%%%%%%%%%%%%%%%%%%%%%%%%%%%%%%%%%%%%%%%%%%%%%%%%%%%%%%%%%%%%%%%%%%%%%%%%%%%%%%%%%
%                                                                                                                %
%  SECTION : INTRODUCTION                                                                                        %
%                                                                                                                %
%%%%%%%%%%%%%%%%%%%%%%%%%%%%%%%%%%%%%%%%%%%%%%%%%%%%%%%%%%%%%%%%%%%%%%%%%%%%%%%%%%%%%%%%%%%%%%%%%%%%%%%%%%%%%%%%%%
\section{Introduction}

In 1974, Jacques Tits published his famous lecture notes ``Buildings of Spherical Type and Finite BN-Pairs'' \cite{T74},
in which he classified all spherical buildings of rank at least $3$.
The main motivation for studying buildings is given by the following quote, which we have taken from \cite[p.\@~V]{T74}.
\begin{quote}
	\em
	The origin of the notions of buildings and BN-pairs lies in an attempt to give a systematic procedure for the geometric
	interpretation of the semi-simple Lie groups and, in particular, the exceptional groups.
\end{quote}
It took another $26$ years before the analogous result in lower rank, namely the classification of spherical buildings of rank $2$ satisfying
the Moufang property, was finished by Jacques Tits and Richard Weiss \cite{TW}.
There is no doubt that the hardest part in the whole classification is precisely where the exceptional quadrangles turn up,
and in fact, there is an ongoing effort to try to understand the structure of these exceptional quadrangles
and the corresponding rank $2$ forms of the exceptional linear algebraic groups.

A first attempt to provide an algebraic framework to describe the Moufang quadrangles was given by the second author \cite{D},
who introduced {\em quadrangular systems}, an algebraic structure consisting of a pair of groups intertwined in a very
delicate way, with no less than $20$ defining axioms.
Unfortunately, this algebraic structure is not very well suited for getting a deeper algebraic understanding
of the specific examples, in particular the exceptional ones;
its main purpose is to have a uniform algebraic structure to describe all possible Moufang quadrangles.

A second attempt, which is more focused on the exceptional Moufang quadrangles, was provided by Richard Weiss \cite{W},
who developed a theory of {\em quadrangular algebras}.
They describe the algebraic structures needed to construct the exceptional Moufang quadrangles
as a generalization of pseudo-quadratic spaces.

Inspired by discussions that we had with Skip Garibaldi, we became aware of the existence of a large class of
non-associative algebras called {\em structurable algebras}, which have been introduced by Bruce Allison in 1978 \cite{A1}
in the context of exceptional Lie algebras.
These structurable algebras were known to be related to yet another class of algebraic structures, known
as {\em Freudenthal triple systems}, introduced by Kurt Meyberg in 1968 \cite{Me}.
We immediately point out one drawback of both algebraic systems: at the moment, they are only defined over fields of characteristic
different from $2$, and very often (but this seems less essential) also different from $3$.

It turns out that every quadrangular algebra over a field of good characteristic can be made into a Freudenthal triple system,
and consequently also into a structurable algebra;
see Theorem~\ref{th:Q-main} below.
Even in the case of the exceptional quadrangular algebras (i.e.\@ those of type $E_6$, $E_7$ and $E_8$) these structurable algebras
are well understood and can be nicely described; see Theorem~\ref{th:E-main} below.
It is our hope that these new insights will lead to a better understanding of the corresponding Moufang quadrangles.

We point out that the algebraic structures that we obtain, are essentially describing the structure of the (non-abelian) rank one residues;
nevertheless, they are related with the module structure (arising from the rank two structure) in a surprisingly simple fashion
(see Theorem~\ref{th: module FTS} and Corollary~\ref{co: V-mod} below).

\subsubsection*{Organization of the paper}

The paper is organized as follows.
In section~\ref{se:moufpol}, we first recall the notion of a Moufang polygon,
and we explain how many of the examples are associated to linear algebraic groups of relative rank two.
The reader can safely skip this section if he wishes; it is not directly used in the rest of the paper,
but explains why the study of quadrangular algebras is of interest.
In section~\ref{se:algknown}, we recall some known facts about the different types of algebraic structures that
we will encounter, namely quadrangular algebras, \FTS s, and structurable algebras.

In section~\ref{se:FTS}, we explain the connection between quadrangular algebras and \FTS s.
In the main section~\ref{se:strQA}, we explain how every quadrangular algebra can be made into a structurable algebra;
we have to use some explicit form of Galois descent in order to arrive at an exact description of those algebras.
We then have a closer look at the case of pseudo-quadratic quadrangular algebras in section~\ref{se:pqQA},
and at the case of quadrangular algebras of type $E_6$, $E_7$ and $E_8$ in the final section~\ref{se:exQA}.

\subsubsection*{Acknowledgments}

We thank Richard Weiss for several very interesting and stimulating discussions on the exceptional Moufang quadrangles,
and in particular we were inspired by an unpublished manuscript of him that we could use in the proof of Lemma~\ref{le:E1}.
We are grateful to Detlev Hoffmann for pointing out to us why the Arason invariant determines the quadratic forms of type $E_8$
up to similarity (see Remark~\ref{rem:arason}).
We also express our gratitude to Skip Garibaldi for many fruitful discussions, and in particular for making us aware
of the theory of structurable algebras, which was completely unknown to us.

Finally, we are extremely thankful to the referee who did an absolutely amazing job by carefully reading an
earlier version of the paper from A~to~Z,
and suggesting many relevant improvements.

\begin{small}
\setcounter{tocdepth}{2}
\renewcommand{\contentsname}{}
\tableofcontents
\end{small}

%%%%%%%%%%%%%%%%%%%%%%%%%%%%%%%%%%%%%%%%%%%%%%%%%%%%%%%%%%%%%%%%%%%%%%%%%%%%%%%%%%%%%%%%%%%%%%%%%%%%%%%%%%%%%%%%%%
%                                                                                                                %
%  SECTION : MOUFANG POLYGONS                                                                                    %
%                                                                                                                %
%%%%%%%%%%%%%%%%%%%%%%%%%%%%%%%%%%%%%%%%%%%%%%%%%%%%%%%%%%%%%%%%%%%%%%%%%%%%%%%%%%%%%%%%%%%%%%%%%%%%%%%%%%%%%%%%%%
\section{Moufang polygons}\label{se:moufpol}

A {\em Moufang polygon} is a notion from incidence geometry introduced by Jacques Tits.
We only give a brief summary of the theory of Moufang polygons, and we refer to \cite{TW} or \cite{DV} for more details.
The importance will immediately become clear in Theorem~\ref{th:mpol} below.

\subsection{Definitions}

A {\em generalized $n$-gon} $\Gamma$ is a connected bipartite graph with diameter $n$ and girth $2n$, where $n \geq 2$.
If we do not want to specify the value of $n$, then we call this a {\em generalized polygon}.
We call a generalized polygon {\em thick} if every vertex has at least $3$ neighbors.
A {\em root} in $\Gamma$ is a (non-stammering) path of length $n$ in $\Gamma$.
Observe that the two extremal vertices of such a path are always {\em opposite}, i.e.\@ their distance is equal to the diameter $n$ of $\Gamma$.
An {\em apartment} in $\Gamma$ is a circuit of length $2n$.

Let $\Gamma$ be a thick generalized $n$-gon with $n \geq 3$, and let $\alpha = (x_0,\dots,x_n)$ be a root of $\Gamma$.
Then the group $U_\alpha$ of all automorphisms of $\Gamma$ fixing all neighbors of $x_1, \dots, x_{n-1}$ (called a {\em root group})
acts freely on the set of vertices incident with $x_0$ but different from $x_1$.
If $U_\alpha$ acts transitively on this set (and hence regularly), then we say that $\alpha$ is a {\em Moufang root}.
It turns out that $\alpha$ is a Moufang root if and only if $U_\alpha$ acts regularly on the set of
apartments through $\alpha$.

A {\em Moufang polygon} is a generalized $n$-gon for which every root is Moufang.
We then also say that $\Gamma$ satisfies the {\em Moufang condition}.
The group generated by all the root groups is sometimes called the {\em little projective group} of $\Gamma$.

Moufang polygons have been classified by J.~Tits and R.~Weiss \cite{TW}.
Loosely speaking, the result is the following.
\begin{theorem}[\cite{TW}]\label{th:mpol}
	Let $\Gamma$ be a Moufang $n$-gon.
	Then $n \in \{ 3, 4, 6, 8 \}$.
	Moreover, every Moufang polygon arises from an absolutely simple linear algebraic group of relative rank $2$,
	or from a corresponding classical group or group of mixed type.
\end{theorem}
In particular, every Moufang polygon is of ``algebraic origin'', and in fact, the Moufang polygons provide a useful tool
to help in the understanding of the corresponding groups; this is particularly true for the Moufang polygons arising
from linear algebraic groups of exceptional type.
For instance, the Kneser--Tits problem for groups of type $E_{8,2}^{66}$ has recently been solved using the theory of Moufang polygons \cite{PTW}.

In order to describe a Moufang polygon in terms of algebraic data, we will use so-called {\em root group sequences}.
A root group sequence for a Moufang $n$-gon is a sequence of $n$ root groups, labeled $U_1,\dots,U_n$, together with
{\em commutator relations} describing how elements of two different root groups $U_i$ and $U_j$ commute.
In each case, the commutator of an element of $U_i$ and $U_j$ (with $i<j$) belongs to the group $\langle U_{i+1},\dots, U_{j-1} \rangle$.
The following result is crucial.
\begin{theorem}
	Let $\Gamma$ be a Moufang $n$-gon.
	Then $\Gamma$ is completely determined by the root groups $U_1,\dots,U_n$ together with their commutator relations.
\end{theorem}
\begin{proof}
	See \cite[Chapter 7]{TW}.
\end{proof}
For more details about this procedure, and how the Moufang polygons can be reconstructed from the root group sequences,
we refer to \cite{TW} or to the survey article \cite{DV}.

For each type of Moufang polygons, we will describe an {\em algebraic structure} which will allow us to parametrize the root groups and describe the commutator relations.

\subsection{Algebraic structures for Moufang polygons}

We give an overview of the classification of Moufang polygons and of the algebraic structures involved in this classification.
This section contains more information than we will actually need, but it puts our theory in a broader context,
which is not so easy to find in the existing literature.

\subsubsection{Moufang triangles}\label{sss:triangles}

Every Moufang triangle (i.e.\@ a Moufang projective plane) can be described in terms of an {\em alternative division algebra},
i.e.\@ a division algebra $A$ which is not necessarily associative, but which instead satisfies the weaker identities
\[ a^{-1} (ab) = b = (ba) a^{-1} \quad \text{ for all } a,b \in A \setminus \{ 0 \} . \]
These algebras have been classified by Bruck and Kleinfeld:
either they are associative after all, or they are $8$-dimensional {\em Cayley--Dickson division algebras}, also known as {\em octonion algebras}.

If $A$ is such an alternative division algebra, then we define $U_1 \cong U_2 \cong U_3 \cong (A, +)$;
we denote the explicit isomorphisms from $A$ to $U_i$ by $x_i$, and we call this the {\em parametrization} of the groups $U_i$ by $(A,+)$.
The commutator relations are then given by
\[ [x_1(a), x_3(b)] = x_2(ab) \]
for all $a,b \in A$; note that the other commutators $[U_1, U_2]$ and $[U_2, U_3]$ are trivial.
Every Moufang triangle can be described in this fashion.

If $A$ is a finite-dimensional division algebra of degree $d$, then this Moufang triangle arises from a linear algebraic group of
absolute type $\mathsf{A}_{3d-1}$.
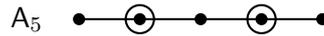
\begin{figure}[ht!]
\begin{center}
\begin{pspicture}(-1,0.2)(4,0.5)
	\psset{unit=.8}
	\rput[r](-.6,0){$\mathsf{A}_5$}
	\pscircle*(0,0){.1}
	\pscircle*(1,0){.1} \pscircle(1,0){.25}
	\pscircle*(2,0){.1}
	\pscircle*(3,0){.1} \pscircle(3,0){.25}
	\pscircle*(4,0){.1}
	\psline(0,0)(4,0)
\end{pspicture}
\end{center}
\caption{Moufang triangle parametrized by a quaternion division algebra}
\end{figure}

If $A$ is infinite-dimensional, then the associated group is no longer an algebraic group, but it can still be viewed as a classical group,
namely $\mathsf{PSL}_3(A)$.

The case where $A$ is an octonion division algebra is exceptional, and arises from a linear algebraic group of absolute type $\mathsf{E}_6$.
\begin{figure}[ht!]
\begin{center}
\begin{pspicture}(-1,0.2)(4,1.0)
	\psset{unit=.8}
	\rput[r](-.6,0){$\mathsf{E}_6$}
	\pscircle*(0,0){.1} \pscircle(0,0){.25}
	\pscircle*(1,0){.1}
	\pscircle*(2,0){.1}
	\pscircle*(3,0){.1}
	\pscircle*(4,0){.1} \pscircle(4,0){.25}
	\pscircle*(2,1){.1}
	\psline(0,0)(4,0)
	\psline(2,0)(2,1)
\end{pspicture}
\end{center}
\caption{Moufang triangle parametrized by an octonion division algebra}
\end{figure}
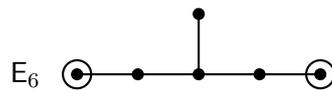

\subsubsection{Moufang hexagons}\label{sss:hexagons}

Every Moufang hexagon can be described in terms of a {\em (quadratic) Jordan division algebra of degree $3$},
or equivalently, by {\em anisotropic cubic norm structures}.
We will not give a precise definition of these algebraic structures since we will not need them explicitly, but we refer
to \cite{DV, KMRT, TW} instead.

We will only mention that if $J$ is such an anisotropic cubic norm structure over a field $K$, then either $J/K$ is a purely inseparable cubic extension,
or $\dim_K J \in \{ 1, 3, 9, 27 \}$.

If $J$ is such an anisotropic cubic norm structure, then we define $U_1 \cong U_3 \cong U_5 \cong (J, +)$ and
$U_2 \cong U_4 \cong U_6 \cong (K, +)$.
The commutator relations can be expressed in terms of the norm, trace and Freudenthal cross product, but their explicit form
is not important for us; we refer to \cite{DV, TW} for more details.

The case where $J = K$ gives rise to the so-called {\em split Cayley hexagon}, which arises from a split linear algebraic group of type $\mathsf{G}_{2}$.
\begin{figure}[ht!]
\begin{center}
\begin{pspicture}(-1,0.2)(1,0.5)
	\psset{unit=.8}
	\rput[r](-.6,0){$\mathsf{G}_2$}
	\pscircle*(-0.1,0){.1} \pscircle(-0.1,0){.25}
	\pscircle*(1.1,0){.1} \pscircle(1.1,0){.25}
	\psline(-0.1,0)(1.1,0)
	\psline(-0.1,0.08)(1.1,0.08)
	\psline(-0.1,-0.08)(1.1,-0.08)
	\psline(0.75,0.3)(0.35,0)(0.75,-0.3)
\end{pspicture}
\end{center}
\caption{Moufang hexagon parametrized by a commutative field}
\end{figure}
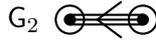

If $J$ is a cubic separable extension field of $K$, then the corresponding Moufang hexagon is the so-called {\em twisted triality hexagon},
which arises from a quasi-split linear algebraic group of type $^{3}\mathsf{D}_{4}$ or $^{6}\mathsf{D}_{4}$.
\begin{figure}[ht!]
\begin{center}
\begin{pspicture}(-1,0.0)(1,0.7)
	\psset{unit=.8}
	\rput[r](-.6,0){$^{3,6}\mathsf{D}_4$}
	\pscircle*(0,0){.1}
	\pscircle*(1,0.5){.1}
	\pscircle*(1,0){.1}
	\pscircle*(1,-0.5){.1}
	\psarc(0.5,0){0.5}{90}{-90}
	\psline(0.5,0.5)(1,0.5)
	\psline(0.5,-0.5)(1,-0.5)
	\psarc(1,0.5){0.25}{0}{180}
	\psarc(1,-0.5){0.25}{180}{0}
	\psline(1.25,0.5)(1.25,-0.5)
	\psline(0.75,0.5)(0.75,-0.5)
	\psline(0,0)(1,0)
	\pscircle(0,0){.25}
\end{pspicture}
\end{center}
\caption{Moufang hexagon parametrized by a cubic extension}
\end{figure}
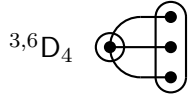

If $J$ is a purely inseparable cubic extension field of $K$, then the corresponding Moufang hexagon arises from a so-called group of mixed type;
such a group is a subgroup (as an abstract group) of an algebraic group of type~$\mathsf{G}_2$, but it is defined over the {\em pair of fields}
$(K, J)$ instead of over a single field.
See, for instance, \cite{T74} for more information on these groups of mixed type.

The next case is where $J$ is a $9$-dimensional $K$-algebra.
There are two cases to distinguish; either $J$ is a central simple cubic cyclic division algebra,
or it is a twisted form of such an algebra, arising from an involution of the second kind on such an algebra.
The resulting Moufang hexagons arise from linear algebraic groups of type $\mathsf{E}_6$ and $^2\mathsf{E}_6$, respectively.
\begin{figure}[ht!]
\begin{center}
\begin{pspicture}(-1,0.2)(4,1.0)
	\psset{unit=.8}
	\rput[r](-.6,0){$\mathsf{E}_6$}
	\pscircle*(0,0){.1}
	\pscircle*(1,0){.1}
	\pscircle*(2,0){.1} \pscircle(2,0){.25}
	\pscircle*(3,0){.1}
	\pscircle*(4,0){.1}
	\pscircle*(2,1){.1} \pscircle(2,1){.25}
	\psline(0,0)(4,0)
	\psline(2,0)(2,1)
\end{pspicture}
\begin{pspicture}(-1.5,-0.3)(3.5,0.6)
	\psset{unit=.8}
	\rput[r](-.6,-.5){$^2\mathsf{E}_6$}
	\pscircle*(0,0){.1} \pscircle(0,0){.25}
	\pscircle*(1,0){.1} \pscircle(1,0){.25}
	\pscircle*(2,.5){.1}
	\pscircle*(2,-.5){.1}
	\pscircle*(3,.5){.1}
	\pscircle*(3,-.5){.1}
	\psline(0,0)(1,0)
	\psarc(1.5,0){0.5}{90}{-90}
	\psline(1.5,0.5)(3,0.5)
	\psline(1.5,-0.5)(3,-0.5)
\end{pspicture}
\end{center}
\caption{Moufang hexagons parametrized by $9$-dimensional cubic norm structures}
\end{figure}
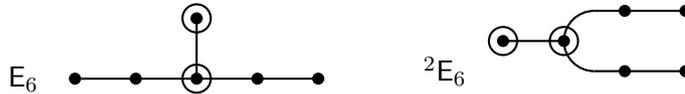

Finally, if $J$ is $27$-dimensional over $K$, then it is an Albert division algebra,
and the resulting Moufang hexagons arise from linear algebraic groups of absolute type $\mathsf{E}_8$.
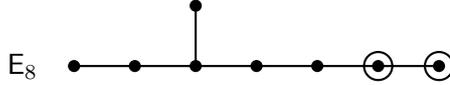
\begin{figure}[ht!]
\begin{center}
\begin{pspicture}(-1,0.2)(4,1.0)
	\psset{unit=.8}
	\rput[r](-.6,0){$\mathsf{E}_8$}
	\pscircle*(0,0){.1}
	\pscircle*(1,0){.1}
	\pscircle*(2,0){.1}
	\pscircle*(3,0){.1}
	\pscircle*(4,0){.1}
	\pscircle*(5,0){.1} \pscircle(5,0){.25}
	\pscircle*(6,0){.1} \pscircle(6,0){.25}
	\pscircle*(2,1){.1}
	\psline(0,0)(6,0)
	\psline(2,0)(2,1)
\end{pspicture}
\end{center}
\caption{Moufang hexagons parametrized by Albert division algebras}
\end{figure}

\subsubsection{Moufang octagons}

The Moufang octagons have a fairly simple structure from an algebraic point of view.
Every Moufang octagon can be described from a commutative field $K$ with $\Char(K) = 2$ equipped with a
{\em Tits endomorphism} $\sigma$, i.e.\@ an endomorphism such that $(x^\sigma)^\sigma = x^2$ for all $x \in K$.
The root groups $U_1, U_3, U_5, U_7$ are parametrized by $(K, +)$, and the root groups $U_2, U_4, U_6, U_8$ are
parametrized by some non-abelian group $T$ with underlying set $K \times K$, and with group operation
\[ (a, b) \cdot (c, d) := (a+c, b+d+a^\sigma c) \quad \text{ for all } a,b,c,d \in K . \]
We do not go into more details, and we again refer to \cite{TW}.
The corresponding groups are the Ree groups of type $^2\mathsf{F}_4$.
\begin{figure}[ht!]
\begin{center}
\begin{pspicture}(-1.5,-0.3)(3.5,0.6)
	\psset{unit=.8}
	\rput[r](-.6,0){$^2\mathsf{F}_4$}
	\pscircle*(1,.5){.1}
	\pscircle*(1,-.5){.1}
	\pscircle*(2,.5){.1}
	\pscircle*(2,-.5){.1}
	\psarc(.8,0){0.44}{90}{-90}
	\psarc(.8,0){0.56}{90}{-90}
	\psline(.8,.44)(1,.44)
	\psline(.8,.56)(1,.56)
	\psline(.8,-.44)(1,-.44)
	\psline(.8,-.56)(1,-.56)
	\psline(1,0.5)(2,0.5)
	\psline(1,-0.5)(2,-0.5)
	\psarc(1,0.5){0.25}{0}{180}
	\psarc(1,-0.5){0.25}{180}{0}
	\psline(1.25,0.5)(1.25,-0.5)
	\psline(0.75,0.5)(0.75,-0.5)
	\psarc(2,0.5){0.25}{0}{180}
	\psarc(2,-0.5){0.25}{180}{0}
	\psline(2.25,0.5)(2.25,-0.5)
	\psline(1.75,0.5)(1.75,-0.5)
\end{pspicture}
\end{center}
\caption{Moufang octagons}
\end{figure}
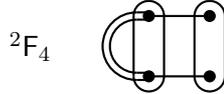

\subsubsection{Moufang quadrangles}\label{sss:quadrangles}

We finally come to the most involved case in the classification, which is the case of the Moufang quadrangles.
In principle, it is possible to define a single algebraic structure to describe all possible Moufang quadrangles;
this gives rise to the so-called {\em quadrangular systems} which have been introduced by the second author \cite{D}.
These structures, however, have some disadvantages from an algebraic point of view; most notably, the definition does not
mention an underlying field of definition (although it is possible to construct such a field from the data), and
the axiom system looks very wild and complicated, with no less than $20$ defining identities.

Instead, we will follow the original classification as given by Tits and Weiss in \cite{TW}, distinguishing six different
(non-disjoint) classes:
\begin{compactenum}[(1)]
\item Moufang quadrangles of indifferent type;
\item Moufang quadrangles of quadratic form type;
\item Moufang quadrangles of involutory type;
\item Moufang quadrangles of pseudo-quadratic form type;
\item Moufang quadrangles of type $E_6$, $E_7$ and $E_8$;
\item Moufang quadrangles of type $F_4$.
\end{compactenum}
The Moufang quadrangles of types (2)--(4) are often called {\em classical}, those of type (5) and (6) are called {\em exceptional}
and those of type (1) are of {\em mixed type}.
Since the Moufang quadrangles of type (1) and (6) only exist over fields of characteristic two, and moreover are not directly related
to rank two forms of algebraic groups, we will exclude those two classes from our further discussion.

\paragraph{Moufang quadrangles of quadratic form type}

Moufang quadrangles of quadratic form type are determined by an anisotropic quadratic form $q \colon V \to K$, where $V$ is an
arbitrary (possibly infinite-dimensional) vector space over some commutative field $K$.
The root groups $U_1$ and $U_3$ are parametrized by $(K, +)$ and the root groups $U_2$ and $U_4$ are parametrized by $(V, +)$;
the commutator relations will involve the quadratic form $q$ and its corresponding bilinear form $f$.
If $d = \dim_K V$ is finite, then these Moufang quadrangles arise from algebraic groups;
they are of absolute type $\mathsf{B}_{\ell+2}$ if $d = 2\ell + 1$ is odd, and of type $\mathsf{D}_{\ell+2}$ if $d = 2\ell$ is even.
\begin{figure}[ht!]
\begin{center}
\begin{pspicture}(-1,0.2)(5,0.5)
	\psset{unit=.8}
	\rput[r](-.6,0){$\mathsf{B}_{\ell+2}$}
	\pscircle*(0,0){.1} \pscircle(0,0){.25}
	\pscircle*(1,0){.1} \pscircle(1,0){.25}
	\pscircle*(2,0){.1}
	\pscircle*(4,0){.1}
	\pscircle*(5,0){.1}
	\pscircle*(6,0){.1}
	\psline(0,0)(2.3,0)
	\psline[linestyle=dotted](2.3,0)(3.7,0)
	\psline(3.7,0)(5,0)
	\psline(4.9,0.06)(6.1,0.06)
	\psline(4.9,-0.06)(6.1,-0.06)
	\psline(5.25,0.3)(5.65,0)(5.25,-0.3)
\end{pspicture}
\end{center}
\caption{Moufang quadrangles parametrized by quadratic forms (odd dimension)}
\end{figure}
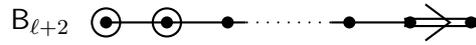

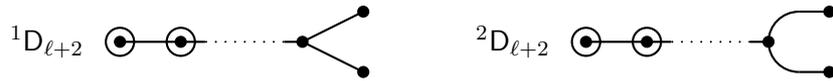
\begin{figure}[ht!]
\begin{center}
\begin{pspicture}(-1.5,0.0)(4,0.5)
	\psset{unit=.8}
	\rput[r](-.6,0){$^1\mathsf{D}_{\ell+2}$}
	\pscircle*(0,0){.1} \pscircle(0,0){.25}
	\pscircle*(1,0){.1} \pscircle(1,0){.25}
	\pscircle*(3,0){.1}
	\pscircle*(4,0.5){.1}
	\pscircle*(4,-0.5){.1}
	\psline(0,0)(1.4,0)
	\psline[linestyle=dotted](1.4,0)(2.7,0)
	\psline(2.7,0)(3,0)
	\psline(4,-0.5)(3,0)(4,0.5)
\end{pspicture}
\begin{pspicture}(-2,0.0)(3.5,0.5)
	\psset{unit=.8}
	\rput[r](-.6,0){$^2\mathsf{D}_{\ell+2}$}
	\pscircle*(0,0){.1} \pscircle(0,0){.25}
	\pscircle*(1,0){.1} \pscircle(1,0){.25}
	\pscircle*(3,0){.1}
	\pscircle*(4,0.5){.1}
	\pscircle*(4,-0.5){.1}
	\psline(0,0)(1.4,0)
	\psline[linestyle=dotted](1.4,0)(2.7,0)
	\psline(2.7,0)(3,0)
	\psarc(3.5,0){0.5}{90}{-90}
	\psline(3.5,.5)(4,.5)
	\psline(3.5,-.5)(4,-.5)
\end{pspicture}
\end{center}
\caption{Moufang quadrangles parametrized by quadratic forms (even dimension)}
\end{figure}

\paragraph{Moufang quadrangles of involutory type}

Moufang quadrangles of involutory type are determined by a (skew) field $K$ equipped%
\footnote{If $\Char(K) = 2$, some more data are required, but this is not relevant for our purposes.}
with an involution $\sigma$.
The root groups $U_2$ and $U_4$ are parametrized by $(K, +)$ and the root groups $U_1$ and $U_3$ are parametrized by $(\Fix_K(\sigma), +)$.
If $K$ is finite-dimensional over its center, of degree $d$, then these Moufang quadrangles arise from algebraic groups;
they are of absolute type $^2\mathsf{A}_{4d-1}$ if the involution is of the second kind,
and they are of absolute type $^1\mathsf{D}_{2d}$ if the involution is of the first kind.

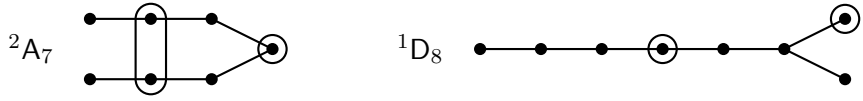
\begin{figure}[ht!]
\begin{center}
\begin{pspicture}(-1.5,0.0)(3,1)
	\psset{unit=.8}
	\rput[r](-.6,0){$^2\mathsf{A}_{7}$}
	\pscircle*(0,-.5){.1} \pscircle*(0,.5){.1}
	\pscircle*(1,-.5){.1} \pscircle*(1,.5){.1}
	\pscircle*(2,-.5){.1} \pscircle*(2,.5){.1}
	\pscircle*(3,0){.1} \pscircle(3,0){.25}
	\psline(0,-.5)(2,-.5)
	\psline(0,.5)(2,.5)
	\psline(2,-0.5)(3,0)(2,0.5)
	\psarc(1,0.5){0.25}{0}{180}
	\psarc(1,-0.5){0.25}{180}{0}
	\psline(1.25,0.5)(1.25,-0.5)
	\psline(0.75,0.5)(0.75,-0.5)
\end{pspicture}
\begin{pspicture}(-2,0.0)(5.5,1)
	\psset{unit=.8}
	\rput[r](-.6,0){$^1\mathsf{D}_{8}$}
	\pscircle*(0,0){.1}
	\pscircle*(1,0){.1}
	\pscircle*(2,0){.1}
	\pscircle*(3,0){.1} \pscircle(3,0){.25}
	\pscircle*(4,0){.1}
	\pscircle*(5,0){.1}
	\pscircle*(6,0.5){.1} \pscircle(6,0.5){.25}
	\pscircle*(6,-0.5){.1}
	\psline(0,0)(5,0)
	\psline(6,-0.5)(5,0)(6,0.5)
\end{pspicture}
\end{center}
\caption{Some Moufang quadrangles of involutory type}
\end{figure}

\paragraph{Moufang quadrangles of pseudo-quadratic form type}

Moufang quadrangles of pseudo-quadratic form type are determined%
\footnote{Again, the situation is slightly more complicated in characteristic two, but we omit the details.}
by an anisotropic pseudo-quadratic space.
\begin{definition}[{\cite[Definition 1.16]{W}}]\label{def:pqs}
A pseudo-quadratic space is a quintuple $(L,\si,X,h,\pi)$ where
\begin{compactenum}[(i)]
    \item
	$L$ is a skew field;
    \item
	$\si$ is an involution of $L$, and
	we let $K := \Fix_L(\sigma)$;
    \item
	$X$ is a right vector space over $L$;
    \item
	$h \colon X\times X \rightarrow L$ is a {\em skew-hermitian form}, i.e.
	\begin{compactitem}
	\item $h$ is bi-additive and $h(x,yu)=h(x,y)u$,  and
	\item $h(x,y)^\si=-h(y,x)$,
	\end{compactitem}
	for all $x,y \in X$ and all $u\in L$;
    \item
	$\pi$ is a {\em pseudo-quadratic form} from $X$ to~$L$, i.e.
	\begin{compactitem}
	\item $\pi(x+y)\equiv \pi(x)+\pi(y)+h(x,y) \mod K$, and
	\item $\pi(xu)\equiv u^\si\pi(x) u \mod K$,
	\end{compactitem}
	for all $x,y \in X$ and all $u \in L$.
\end{compactenum}
A pseudo-quadratic space $(L,\si,X,h,\pi)$ is called {\em anisotropic} if
\[\pi(x)\equiv0 \mod K \ \text{ only if }x=0.\]
\end{definition}
\begin{remark}
	If $\Char(L) \neq 2$, then the pseudo-quadratic form $\pi$ is completely determined (modulo $K$) by the skew-hermitian form $h$.
	We have nevertheless decided to include the pseudo-quadratic form in the definition,
	because this form will play an important role in the sequel.
\end{remark}

The root groups $U_2$ and $U_4$ are parametrized by $(L, +)$; the root groups $U_1$ and $U_3$ are parametrized by
some non-abelian group with underlying set $X \times K$, and both the group operation and the commutator relations
involve the maps $h$ and $\pi$.
If $L$ is finite-dimensional over its center, of degree $d$, and $X$ is finite-dimensional over $L$,
then these Moufang quadrangles arise from algebraic groups.
If the involution is of the second kind, they are of absolute type $^2\mathsf{A}_{\ell}$.
If the involution is of the first kind, they are of absolute type $\mathsf{C}_{\ell}$, $^1\mathsf{D}_{\ell}$ or $^2\mathsf{D}_{\ell}$.

\begin{figure}[ht!]
\begin{center}
\begin{pspicture}(-1.5,-.5)(4.5,1)
	\psset{unit=.8}
	\rput[r](-.6,0){$^2\mathsf{A}_{5}$}
	\pscircle*(0,-.5){.1} \pscircle*(0,.5){.1}
	\pscircle*(1,-.5){.1} \pscircle*(1,.5){.1}
	\pscircle*(2,0){.1}
	\psline(0,-.5)(1,-.5)
	\psline(0,.5)(1,.5)
	\psline(1,-0.5)(2,0)(1,0.5)
	\psarc(1,0.5){0.25}{0}{180}
	\psarc(1,-0.5){0.25}{180}{0}
	\psline(1.25,0.5)(1.25,-0.5)
	\psline(0.75,0.5)(0.75,-0.5)
	\psarc(0,0.5){0.25}{0}{180}
	\psarc(0,-0.5){0.25}{180}{0}
	\psline(0.25,0.5)(0.25,-0.5)
	\psline(-.25,0.5)(-.25,-0.5)
\end{pspicture}
\begin{pspicture}(-1,-.5)(4,0.5)
	\psset{unit=.8}
	\rput[r](-.6,0){$\mathsf{C}_{4}$}
	\pscircle*(0,0){.1}
	\pscircle*(1,0){.1} \pscircle(1,0){.25}
	\pscircle*(2,0){.1}
	\pscircle*(3,0){.1} \pscircle(3,0){.25}
	\psline(0,0)(2,0)
	\psline(1.9,0.06)(3.1,0.06)
	\psline(1.9,-0.06)(3.1,-0.06)
	\psline(2.75,0.3)(2.35,0)(2.75,-0.3)
\end{pspicture}

\begin{pspicture}(-1.5,0.0)(4.5,1)
	\psset{unit=.8}
	\rput[r](-.6,0){$^1\mathsf{D}_{6}$}
	\pscircle*(0,0){.1}
	\pscircle*(1,0){.1} \pscircle(1,0){.25}
	\pscircle*(2,0){.1}
	\pscircle*(3,0){.1} \pscircle(3,0){.25}
	\pscircle*(4,0.5){.1}
	\pscircle*(4,-0.5){.1}
	\psline(0,0)(3,0)
	\psline(4,-0.5)(3,0)(4,0.5)
\end{pspicture}
\begin{pspicture}(-1,0.0)(4,1)
	\psset{unit=.8}
	\rput[r](-.6,0){$^2\mathsf{D}_{6}$}
	\pscircle*(0,0){.1}
	\pscircle*(1,0){.1} \pscircle(1,0){.25}
	\pscircle*(2,0){.1}
	\pscircle*(3,0){.1} \pscircle(3,0){.25}
	\pscircle*(4,0.5){.1}
	\pscircle*(4,-0.5){.1}
	\psline(0,0)(3,0)
	\psarc(3.5,0){0.5}{90}{-90}
	\psline(3.5,.5)(4,.5)
	\psline(3.5,-.5)(4,-.5)
\end{pspicture}
\end{center}
\caption{Some Moufang quadrangles of pseudo-quadratic form type}
\end{figure}
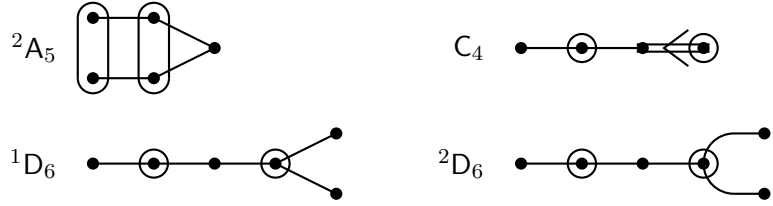

\paragraph{Moufang quadrangles of type $E_6$, $E_7$ and $E_8$}

Moufang quadrangles of type $E_6$, $E_7$ and $E_8$ are exceptional, and always arise from algebraic groups.
The explicit construction of these Moufang quadrangles is very complicated, and one of the goals of our paper is
precisely to get a better understanding of these exceptional quadrangles.
We refer to \cite[Chapter 12 and 13]{TW} or \cite[Chapter 10]{W}, or also to \cite[section 4.3.5]{DV}, for the precise construction;
many of its properties will be captured in the definition of a quadrangular algebra that we will recall in section~\ref{ss:QA} below.
In section~\ref{par: E6E7E8}, we will give the necessary details to make the connection with structurable algebras.

The corresponding Tits indices are as follows.
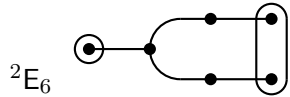
\begin{figure}[ht!]
\begin{center}
\begin{pspicture}(-1.5,-0.3)(3.5,0.6)
	\psset{unit=.8}
	\rput[r](-.6,-.5){$^2\mathsf{E}_6$}
	\pscircle*(0,0){.1} \pscircle(0,0){.25}
	\pscircle*(1,0){.1}
	\pscircle*(2,.5){.1}
	\pscircle*(2,-.5){.1}
	\pscircle*(3,.5){.1}
	\pscircle*(3,-.5){.1}
	\psline(0,0)(1,0)
	\psarc(1.5,0){0.5}{90}{-90}
	\psline(1.5,0.5)(3,0.5)
	\psline(1.5,-0.5)(3,-0.5)
	\psarc(3,0.5){0.25}{0}{180}
	\psarc(3,-0.5){0.25}{180}{0}
	\psline(3.25,0.5)(3.25,-0.5)
	\psline(2.75,0.5)(2.75,-0.5)
\end{pspicture}
\end{center}
\caption{Moufang quadrangles of type $E_6$\label{fig:E6}}
\end{figure}

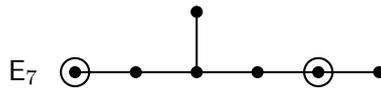
\begin{figure}[ht!]
\begin{center}
\begin{pspicture}(-1,0.2)(4,1.0)
	\psset{unit=.8}
	\rput[r](-.6,0){$\mathsf{E}_7$}
	\pscircle*(0,0){.1} \pscircle(0,0){.25}
	\pscircle*(1,0){.1}
	\pscircle*(2,0){.1}
	\pscircle*(3,0){.1}
	\pscircle*(4,0){.1} \pscircle(4,0){.25}
	\pscircle*(5,0){.1}
	\pscircle*(2,1){.1}
	\psline(0,0)(5,0)
	\psline(2,0)(2,1)
\end{pspicture}
\end{center}
\caption{Moufang quadrangles of type $E_7$\label{fig:E7}}
\end{figure}

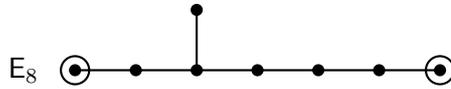
\begin{figure}[ht!]
\begin{center}
\begin{pspicture}(-1,0.2)(4,1.0)
	\psset{unit=.8}
	\rput[r](-.6,0){$\mathsf{E}_8$}
	\pscircle*(0,0){.1} \pscircle(0,0){.25}
	\pscircle*(1,0){.1}
	\pscircle*(2,0){.1}
	\pscircle*(3,0){.1}
	\pscircle*(4,0){.1}
	\pscircle*(5,0){.1}
	\pscircle*(6,0){.1} \pscircle(6,0){.25}
	\pscircle*(2,1){.1}
	\psline(0,0)(6,0)
	\psline(2,0)(2,1)
\end{pspicture}
\end{center}
\caption{Moufang quadrangles of type $E_8$}
\end{figure}

%%%%%%%%%%%%%%%%%%%%%%%%%%%%%%%%%%%%%%%%%%%%%%%%%%%%%%%%%%%%%%%%%%%%%%%%%%%%%%%%%%%%%%%%%%%%%%%%%%%%%%%%%%%%%%%%%%
%                                                                                                                %
%  SECTION : QUADRANGULAR ALGEBRAS                                                                               %
%                                                                                                                %
%%%%%%%%%%%%%%%%%%%%%%%%%%%%%%%%%%%%%%%%%%%%%%%%%%%%%%%%%%%%%%%%%%%%%%%%%%%%%%%%%%%%%%%%%%%%%%%%%%%%%%%%%%%%%%%%%%
\section{Quadrangular algebras, \FTS s and structurable algebras}\label{se:algknown}

In this section, we assemble some known facts about the different kinds of algebraic structures that we will need,
namely quadrangular algebras, \FTS s and structurable algebras.

\subsection{Quadrangular algebras}\label{ss:QA}

In order to try to have a better understanding of the exceptional Moufang quadrangles, i.e.\@ those of type $E_6$, $E_7$, $E_8$ and $F_4$,
Richard Weiss designed an algebraic structure, called a {\em quadrangular algebra} \cite{W}.
These algebras are a generalization of (some) pseudo-quadratic spaces, where, in some sense, the structure of the underlying skew field is lost,
but is replaced by some weaker identities, comparable to the replacement of the associativity by some weaker identities in the definition of
an alternative algebra.
Unfortunately, the definition involves a large number of maps and identities between these maps, that make it rather difficult to have
a deep understanding of the corresponding structures.
Moreover, giving an {\em explicit} construction of the quadrangular algebras related to the exceptional Moufang quadrangles,
involves a very delicate way of introducing coordinates and defining these maps in terms of these coordinates.
We hope that our understanding in terms of structurable algebras will eventually allow to give an explicit construction
avoiding these coordinates, and thereby giving much more insight.

The definition of quadrangular algebras is significantly simpler over fields of characteristic different from $2$;
since this is the only case we will be dealing with in this paper, we will restrict to this case,
and we refer the interested reader to \cite{W} for the general definition.

\begin{definition}\label{def:quad}
A {\em quadrangular algebra} of characteristic different from $2$ is an $8$-tuple $(K,L,q,1,X,\cdot,h,\theta)$, where
\begin{compactenum}[(i)]
    \item
	$K$ is a commutative field with $\Char(K) \neq 2$,
    \item
	$L$ is a $K$-vector space,
    \item
	$q$ is an anistropic quadratic form from $L$ to $K$,
    \item
	$1 \in L$ is a {\em base point} for $q$, i.e.\@ an element such that $q(1) = 1$,
    \item
	$X$ is a non-trivial $K$-vector space,
    \item
	$(a,v) \mapsto a \cdot v$ is a map from $X \times L$ to $X$ (usually denoted simply by juxtaposition),
    \item
	$h$ is a map from $X \times X$ to $L$, and
    \item
	$\theta$ is a map from $X \times L$ to $L$,
\end{compactenum}
satisfying the following axioms, where
\begin{align*}
	&f \colon L \times L \to K \colon (a,b) \mapsto f(a,b) := q(a+b) - q(a) - q(b) \,; \\
	&\sigma \colon L \to L \colon v \mapsto f(1,v) 1 - v \,; \\
	&v^{-1} := v^\sigma / q(v) \,.
\end{align*}
\begin{compactitem}
    \item[(A1)]
	The map $\cdot$ is $K$-bilinear.
    \item[(A2)]
	$a \cdot 1 = a$ for all $a \in X$.
    \item[(A3)]
	$(av)v^{-1} = a$ for all $a \in X$ and all $v \in L^*$.
    \medskip
    \item[(B1)]
	$h$ is $K$-bilinear.
    \item[(B2)]
	$h(a,bv)=h(b,av)+f(h(a,b),1)v$ for all $a,b \in X$ and all $v \in L$.
    \item[(B3)]
	$f(h(av,b),1) = f(h(a,b),v)$ for all $a,b \in X$ and all $v \in L$.
   \medskip
    \item[(C)]
	$\theta(a,v) = \tfrac{1}{2} h(a, av)$.
    \medskip
    \item[(D1)]
	Let $\pi(a) = \theta(a,1)$ for all $a \in X$. Then $a \theta(a, v) = (a\pi(a))v$
	for all $a \in X$ and all $v \in L$.
    \item[(D2)]
	$\pi(a) \equiv 0 \pmod{K}$ if and only if $a = 0$ (where $K$ has been identified
	with its image under the map $t \mapsto t\cdot 1$ from $K$ to $L$).
\end{compactitem}
Moreover, we define a map $g \colon X \times X \to K$ by
\[ g(a,b) := \tfrac{1}{2} f(h(a,b), 1) \]
for all $a,b \in X$.
\end{definition}

Quadrangular algebras have been classified by Richard Weiss, and over fields of characteristic different from two, the
result can be summarized as follows.
\begin{theorem}[\cite{W}]
	Let $\Omega$ be a quadrangular algebra over a field $K$ with $\Char(K) \neq 2$.
	Then either $\Omega$ is an anisotropic pseudo-quadratic space over a quadratic extension $E/K$ or over a quaternion division algebra $Q/K$
	equipped with the standard involution,
	or $\Omega$ it is of type $E_6, E_7$ or $E_8$. 
\end{theorem}
\begin{proof}
	See \cite[Theorem 3.1, Theorem 3.2 and Proposition 3.14]{W}.
\end{proof}

\begin{remark}\label{rem:module}
	If $q$ is a quadratic form from $L$ to $K$, with base point $1 \in L$, then the {\em Clifford algebra of $q$ with basepoint $1$} is
	defined as
	\[ C(q, 1) := T(L) / \langle u \otimes u^\sigma - q(u) \cdot 1 \rangle , \]
	where $T(L)$ is the tensor algebra of $L$, and where $\sigma$ is defined as in Definition~\ref{def:quad}.
	In \cite[(12.51)]{TW} it is shown that $C(q, 1) \cong C_0(q)$, the even Clifford algebra of $q$.
	The notion of a Clifford algebra with base point was introduced by Jacobson and McCrimmon in \cite{JM};
	see also \cite[Chapter 12]{TW} for more details.
	Since $q$ is anisotropic, axioms (A1)--(A3) say precisely that $X$ is a $C(q, 1)$-module.
\end{remark}

\begin{remark}
The definition of $g$ we have used is as in \cite{W}.
This is {\em not} the same definition as in \cite[Chapter 13]{TW}, where
$g(a,b)=\tfrac{1}{2}f(h(b,a),1)=-\tfrac{1}{2}f(h(a,b),1)$.
See also \cite[Remark (viii), p.~7]{W}.
\end{remark}

\begin{remark}
	The definition of quadrangular algebras over fields $K$ with $\Char(K) = 2$ involves four more axioms (C1)--(C4) which replace the axiom (C)
	in Definition~\ref{def:quad} above (and define the map $g$ in a different way).
	One of these axioms, (C4), is considerably more complicated than the other axioms, but
	it can be shown that the axioms (C1)--(C4) are superfluous over fields $K$
	with $\Char(K) \neq 2$ in the sense that axiom (C) implies (C1)--(C4).
	See \cite[Remark 4.8]{W} for more details.
\end{remark}

We will use the following formulas in the sequel.
\begin{theorem}\label{eign:quadr} 
Let $(K,L,q,1,X,\cdot,h,\theta)$ be a quadrangular algebra, with $\Char(K)\neq2$.
For all $a,b\in X$ and all $u,v\in L$ we have that
\begin{compactenum}[\rm (i)]\itemsep.5ex
\item $h(a,b)=-h(b,a)^\si$, % \textup{ (3.6)}
\item $f(h(a,bv),1)=f(h(a,b),v^\si)$, % \textup{ (3.7)}
\item $(au)v=-(av^\si)u^\si+af(u,v^\si)$, % \textup{ (3.8)}
\item $h(a\pi(a),b)+\theta(a,h(a,b))=0$,
\item $\theta(av,w) = \theta(a,w^\sigma)^\sigma q(v) - f(w,v^\sigma)\theta(a,v)^\sigma + f(\theta(a,v),w^\sigma)v^\sigma$.
\end{compactenum}

\end{theorem} 
\begin{proof}
Identities (i)--(iii) are precisely \cite[(3.6), (3.7) and (3.8)]{W}.
Identity (iv) is identity (e) in the proof of \cite[(13.67)]{TW}; the proof holds without any change in the pseudo-quadratic case as well.
Identity (v) is precisely axiom (C4) in \cite[Definition 1.17]{W}, taking into account that the map $\phi$ occuring in this axiom is trivial
by \cite[Proposition 4.5]{W}.
\end{proof}

\begin{comment}Let $v\in L$, we have that
\begin{align*}
f(h(a\pi(a),b),v)&=f(h(a\pi(a)v,b),1)\\
&=f(h(a\theta(a,v),b),1)\\
&=f(h(a,b),\theta(a,v))\\
&=-f(\theta(a,h(a,b)),v)(4.9.iii)
\end{align*}
Since this equality holds for arbitrary $v\in L$ we have that $h(a\pi(a),b)+\theta(a,h(a,b))=0$.
\end{comment}

\subsection{\FTS s}\label{ss:FTS}

\begin{definition}\label{def: FTS}
A {\em Freudenthal triple system} $(V,b,t)$ is a vector space $V$ over a field $K$ of characteristic not 2 or 3, endowed with a trilinear symmetric product 
\[ t \colon V\times V\times V \rightarrow V \colon (x,y,z) \mapsto t(x,y,z) =: xyz \]
and a skew symmetric bilinear form
\[ b \colon V\times V\rightarrow K \colon (x,y)\mapsto b(x,y) =: \langle x,y \rangle \]
such that
\begin{compactenum}[(i)]
\item the map $(x,y,z,w)\mapsto\langle x,yzw\rangle$ is a nonzero symmetric 4-linear form;
\item$(xxx)xy= \langle y,x\rangle xxx+\langle y,xxx\rangle x\quad \forall x,y\in V$.
\end{compactenum}
When it is clear which triple product and skew symmetric form are considered, we do not explicitly mention $b$ and $t$,
but we use juxtaposition and $\langle.,.\rangle$ instead.
\end{definition}

\begin{definition}\label{def: similarity}
Two \FTS s $(V,b,t)$, $(V',b',t')$ over a field $K$ are {\em similar} if there exists a $K$-vector space isomorphism $\psi:V\rightarrow V'$
and $\lambda \in K^*$ such that
\[t'(\psi(x),\psi(y),\psi(z))=\la \psi(t(x,y,z)).\]
In \cite[Lemma 6.6]{F} it is proven that this condition is equivalent with
\[
\begin{cases}
	b'(\psi(x),\psi(y)) = \la b(x,y) \ \text{ and } \\
	b'(\psi(x),t'(\psi(x),\psi(x),\psi(x)))=\la^2 b(x,t(x,x,x)).
\end{cases}
\]
The map $\psi$ is then called a {\em similarity} with {\em multiplier} $\lambda$.
We say that two \FTS s are {\em isometric} if they are similar with $\la=1$; in this case $\psi$ is called an {\em isometry}.
\end{definition}

\begin{definition}
Let $V$ be a Freudenthal triple system.
\begin{compactenum}[(i)]
\item An element $u\in V\setminus \{0\}$ is called {\em strictly regular} if $uVV\subseteq Ku$.
\item A pair of strictly regular elements $u_1, u_2$ is called {\em supplementary} if $\langle u_1,u_2\rangle=1$.
\item $V$ is called {\em reduced} if it contains a strictly regular element.
\item $V$ is called {\em simple} if it does not contain a proper ideal, i.e.\@ a subspace $I\neq{0},V$ such that $IVV\subseteq I$.
\end{compactenum}
\end{definition}

More details can be found in \cite{F}.
We mention a few results that we will use later.

\begin{lemma}\label{nonred and simple}
If the map $x\mapsto\langle x, xxx\rangle$ is anisotropic, the Freudenthal triple system is not reduced and simple.
\end{lemma}
\begin{proof}
Suppose $u\in V$ is strictly regular, then $uuu=ku$ for some $k\in K$, so
\[\langle u,uuu\rangle=k\langle u,u\rangle=0.\]
This implies that $u=0$, so the Freudenthal triple system is not reduced.

In \cite{F} it is shown that a Freudenthal triple system is simple if and only if its bilinear form is nondegenerate.
This is clearly the case, since for every $x \in V \setminus \{ 0 \}$, we have $\langle x, y \rangle \neq 0$ for $y = xxx$.
\end{proof}

\begin{lemma}[{\cite[Corollary 3.4]{F}}]\label{lem: sre in F}
A simple \FTS\ $V$ is reduced if and only there exists $x\in V$ such that $\langle x,xxx\rangle=12 k^2$ for $k\in K^*$.

If this is the case, then
\[ u_1=\frac{1}{2}x +\frac{1}{12k}xxx, \qquad u_2=-\frac{1}{2k}x +\frac{1}{12k^2}xxx \]
is a pair of supplementary strictly regular elements. 
\end{lemma}

\subsection{Structurable algebras}\label{ss:str}

Structurable algebras have been introduced by B.~Allison \cite{A0}
and have been used in the construction of non-split exceptional simple Lie algebras, see for example \cite{A1}.

\begin{definition}A {\em structurable algebra} over a field $K$ of characteristic not~$2$ or $3$ is a unital, not necessarily associative $K$-algebra with involution\footnote{An involution is a $K$-linear map of order 2 such that $\overline{xy}=\overline{y}\,\overline{x}$.} $(\A,\bar{\ })$ such that
\begin{align}\label{struct id}
[V_{x,y},V_{z,w}]=&V_{\{x,y,z\},w}-V_{z,\{y,x,w\}}
\end{align}
for $x,y,z,w \in \A$ where $V_{x,y}z:=\{x,y,z\}:=(x\overline{y})z+ (z\overline{y})x-(z\overline{x})y$.

For all $x,y,z \in \A$, we write $U_{x,y}z:=V_{x,z}y$ and $ U_xy:=U_{x,x}y$.
We will refer to the maps $V_{x,y} \in \End_K(\A)$ as {\em $V$-operators}, and to the maps $U_{x,y} \in \End_K(\A)$ and $U_x \in \End_K(\A)$
as {\em $U$-operators}.
\end{definition}

Structurable algebras are generalizations of both associative algebras with involution and Jordan algebras.
Indeed, the class of structurable algebras with trivial involution is exactly equal to the class of Jordan algebras (in characteristic different from $2$).

If $(\A,\bar{\ })$ is a structurable algebra, then $\A=\mathcal{H}\oplus\Ss$ for
\[ \mathcal{H}=\{h\in \mathcal{A}\mid \overline{h}=h\} \quad \text{and} \quad \Ss=\{s\in \mathcal{A}\mid \overline{s}=-s\}. \]
The dimension of $\Ss$ is called the {\em skew-dimension} of $\A$.
Structurable algebras of skew-dimension 0 are exactly the Jordan algebras. 

\begin{definition}\label{conjugate_inv}
An element $x \in \A$ is called {\em conjugate invertible} if it has a {\em conjugate inverse}, i.e.\@ 
an element $\hat{x}$ such that $V_{\hat{x},x}=V_{x,\hat{x}}=\operatorname{id}$. 
If an element is conjugate invertible, then it has a unique conjugate inverse.
A structurable algebra is a {\em (conjugate) division algebra} if every non-zero element is conjugate invertible.
\end{definition}

Although it is possible to define and study isomorphisms of structurable algebras, it turns out that it is better
to allow the unit element $1$ to be mapped to a different element.
This idea is encapsulated in the notion of an isotopy.
\begin{definition}
Two structurable algebras $(\A,\bar{\ })$ and $(\A',\bar{\ })$ over a field $K$ are {\em isotopic}
if there exists two $K$-vector space isomorphisms $\psi, \chi:\A\rightarrow \A'$ such that
\[\psi(V_{x,y}z)=V_{\psi(x),\chi(y)}\psi(z)\quad\forall x,y,z\in \A.\]
The map $\psi$ is then called an {\em isotopy} between $(\A,\bar{\ })$ and $(\A',\bar{\ })$. 
\end{definition}
\begin{remark}\begin{compactenum}[(i)]
\item If $\psi$ is an isotopy, the map $\chi$ is entirely determined by the map $\psi$.

\item If $\psi$ maps the identity of $\A$ to the identity of $\A'$, $\psi$ is an isomorphism of structurable algebras.

\item If $(\A',\bar{\ })$ is isotopic to $(\A,\bar{\ })$, then there exists a conjugate invertible $u\in \A'$ such that $(\A',\bar{\ })^{\langle u\rangle}$ is isomorphic to $(\A,\bar{\ })$.
For a description of $(\A',\bar{\ })^{\langle u\rangle}$ we refer to \cite[p.\@~188]{CD} or \cite[section 2]{skewdim}.
\end{compactenum}
\end{remark}

\begin{remark}\label{rem:cssa}
Central simple structurable algebras over fields of characteristic different from $2$, $3$ and $5$, have been classified.
They consist of six classes:
\begin{compactenum}[(1)]
\item associative algebras with involution,
\item Jordan algebras,
\item structurable algebras constructed from a hermitian form over an associative algebra with involution (see section~\ref{par: pseudo-quadratic}),
\item forms of structurable matrix algebras (see Example \ref{ex: skew dim}),
\item forms of tensor products of composition algebras,
\item an exceptional $35$-dimensional case.
\end{compactenum}
For a proof we refer to \cite[Theorem 3.8]{smirnov}.

In an associative algebra with involution, an element is conjugate invertible if and only if it is invertible in the usual associative sense.
If $x$ is invertible, its conjugate inverse is equal to $\overline{x}^{-1}$.

In a Jordan algebra, an element is conjugate invertible if and only if it is invertible in the usual Jordan sense.
If $x$ is invertible, its conjugate inverse is equal to its Jordan inverse $x^{-1}$.

\end{remark}

\subsubsection{Structurable algebras of skew-dimension one}

Structurable algebras of skew-dimension one are close to Jordan algebras. Although they are in general not power-associative, concepts from Jordan theory can be adapted to this class of structurable algebras; see \cite{skewdim}.  
It is of particular interest to us that one can give each structurable algebra of skew-dimension one the structure of a simple Freudenthal triple system.  

We give some examples of structurable algebras of skew-dimension one by considering the classification in Remark~\ref{rem:cssa}:
the algebras in case (4) always have skew-dimension one;
those in cases (1) and (3) have skew-dimension one if and only if the associative algebra with involution has skew-dimension one.
The structurable algebras in the remaining cases never have skew-dimension one, except for forms of $E\otimes_K K$ with $E$ a quadratic field extension of the field $K$.

In this section $(\A,\bar{\ })$ is always a structurable algebra of skew-dimension one. Such an algebra is always central simple. 

We fix a non-zero element $s_0\in \Ss$, so $\Ss=Ks_0$.
One can show that $s_0^2=\mu 1$ for some $\mu \in K^*$ and that $s_0(s_0x)=(xs_0)s_0=\mu x$, for all $x\in \A$.

\begin{theorem}[{\cite[Proposition 2.8]{CD}}]\label{th:FTS skew dim}
Let  $(\A,\bar{\ })$ be a structurable algebra of skew-dimension one and let $s_0 \in \Ss$.
The following triple product and bilinear form give $\mathcal{A}$ the structure of a simple Freudenthal triple system:
\begin{align*}
\langle x,y\rangle1 &= (x\overline{y}-y\overline{x})s_0, \\
yzw &= 2\{y,s_0z,w\}-\langle z,w\rangle y-\langle z,y\rangle w-\langle y,w\rangle z.
 \end{align*}
 \end{theorem}
 
 \begin{remark}\label{rem:norm}
 \begin{compactenum}[(i)]
\item If we would choose another generator for $\Ss$, then we would get a Freudenthal triple system that is a scalar mutiple of the one we started with.
  
\item In structurable algebras there exists a generalization of the generic norm in a Jordan algebra, called the {\em conjugate norm}.
For an exact definition see \cite{norms}. 

For algebras of skew-dimension one, the conjugate norm, denoted by~$\nu$, is exactly $\frac{1}{12 \mu}\langle x,xxx\rangle$.
Indeed, one can easily verify that the norm in \cite[Prop.\@~5.4]{norms} and the definition of $\nu$ in \cite[Par.\@~1]{skewdim} are given by the same formulas. 
The quartic map $\nu$ is independent of the choice of~$s_0$.
\end{compactenum}
\end{remark}

In structurable algebras of skew-dimension one there is an easy way to write down the conjugate inverse of an element.
\begin{theorem}[{\cite[Prop.\@~2.11]{CD}}]\label{th: invers in struct}
Let  $(\A,\bar{\ })$ be a structurable algebra of skew-dimension one and let $s_0 \in \Ss$.
Then $x\in \A$ is conjugate invertible (see Definition \ref{conjugate_inv})  if and only if $\nu(x)\neq0$. 
When this is the case, we have
\[\hat{x}=-\frac{1}{3\mu\nu(x)}s_0\{x,s_0x,x\}.\] 
\end{theorem}

On a structurable algebra, we have the notion of similarity; on a \FTS, we have the notion of isotopy.
Theorem \ref{th:FTS skew dim} tells us that a structurable algebra of skew-dimension one is also a \FTS.
The following lemma states that in this case the notions of isotopy and similarity coincide.
\begin{lemma}[{\cite[Proposition 4.11]{G}}]\label{lem: similar and isotopic}
Let  $(\A,\bar{\ })$ and $(\A',\bar{\ })$ be structurable algebras of skew-dimension one.
Consider the corresponding \FTS s as in Theorem \ref{th:FTS skew dim}.
Then $\A$ and $\A'$ are similar as \FTS s if and only if they are isotopic as structurable algebras.
\end{lemma}
\begin{proof}
In \cite{G}, all \FTS s that are considered are $56$-dimensional, but the proof remains valid in arbitrary dimension.
\end{proof} 

We will now discuss an important class of structurable algebras of skew-dimension one, which we will call  structurable matrix algebras. 
\begin{example}[{\cite[Example 1.9]{skewdim}}] \label{ex: skew dim}
We give a short overview of how this class of algebras and the corresponding Freudenthal triple systems are constructed. For more details, see \cite[Example 1.9]{skewdim}.

Let $J$ be a Jordan algebra over a field $K$ constructed from an admissible cubic form $N$ with base point,
or a Jordan algebra constructed from a non-degenerate quadratic form.
(For definitions of these Jordan algebras we refer to \cite[Chapter II.3.3 and II.4]{M}; an admissible cubic form is the same as a Jordan cubic form in \cite[II.4.3]{M}.)
 
The Jordan algebra $J$ is equipped with a non-degenerate trace form $T$ and a symmetric sharp product $\times$(sometimes denoted by $\sharp$). 

We define the {\em structurable matrix algebras} as follows. Fix a constant $\eta\in K$, and define
\[\mathcal{A}= \left\{ \begin{pmatrix}k_1&j_1\\j_2&k_2\end{pmatrix} \Bigm\vert k_1,k_2\in K, \, j_1, j_2 \in J\right\}.\]
For $k_1,k_2,k'_1,k'_2\in K$, $j_1,j_2,j'_1,j'_2\in J$, define the involution and multiplication as follows:
\[\overline{\begin{pmatrix}k_1&j_1\\j_2&k_2\end{pmatrix}}=\begin{pmatrix}k_2&j_1\\j_2&k_1\end{pmatrix},\]
\[\begin{pmatrix}
k_1&j_1\\
j_2&k_2
\end{pmatrix} 
\begin{pmatrix}
k'_1&j'_1\\ 
j'_2&k'_2
\end{pmatrix}=
\begin{pmatrix} 
k_1k'_1+\eta T(j_1,j'_2)&k_1j'_1+k'_2j_1+\eta(j_2\times j'_2)\\
k'_1j_2+k_2j'_2+j_1\times j'_1& k_2k'_2+\eta T(j_2,j'_1)\end{pmatrix}.\]
We denote this structurable matrix algebra by $M(J,\eta)$.

Now let $s_0=\left(\begin{smallmatrix}
1&0\\ 
0&-1
\end{smallmatrix}\right)$;
then the \FTS\  defined in Theorem \ref{th:FTS skew dim} has bilinear product
\[ \left\langle  \begin{pmatrix}k_1&j_1\\j_2&k_2\end{pmatrix}, \begin{pmatrix}k'_1&j'_1\\j'_2&k'_2\end{pmatrix}\right\rangle=k_1k'_2-k_2k'_1+\eta T(j_1, j'_2)-\eta T(j_2,j'_1) , \]
and the conjugate norm is given by
\begin{multline*}
	\nu\begin{pmatrix}k_1&j_1\\j_2&k_2\end{pmatrix} = 4k_1\eta N(j_1)+4k_2\eta^2N(j_2)-4\eta^2 T(j_1^\sharp,j_2^\sharp) \\
	+ \bigl(\eta T(j_1,j_2)-k_1k_2\bigr)^2 .
\end{multline*}
\end{example}

The following theorem shows that each structurable algebra of skew-dimension one is isomorphic to a matrix algebra or becomes isomorphic to a matrix algebra after adjoining $\sqrt{\mu}$ to the base field.
\begin{theorem}[{\cite[Prop.\@~4.5]{CD}}] Let  $(\A,\bar{\ })$ be a structurable algebra of skew-dimension one, and let $s_0 \in \Ss$ with $s_0^2=\mu1$.
Then $(\A,\bar{\ })$ is isomorphic to a matrix algebra $M(J,\eta)$ if and only if $\mu$ is a square in $K$.
\end{theorem}

%%%%%%%%%%%%%%%%%%%%%%%%%%%%%%%%%%%%%%%%%%%%%%%%%%%%%%%%%%%%%%%%%%%%%%%%%%%%%%%%%%%%%%%%%%%%%%%%%%%%%%%%%%%%%%%%%%
%                                                                                                                %
%  SECTION : FREUDENTHAL TRIPLE SYSTEMS                                                                          %
%                                                                                                                %
%%%%%%%%%%%%%%%%%%%%%%%%%%%%%%%%%%%%%%%%%%%%%%%%%%%%%%%%%%%%%%%%%%%%%%%%%%%%%%%%%%%%%%%%%%%%%%%%%%%%%%%%%%%%%%%%%%
\section{Quadrangular algebras and Freudenthal triple systems}\label{se:FTS}

We show that each quadrangular algebra defined over a field of characteristic not 2 or 3 can be given the structure of a Freudenthal triple system. 
It turns out that the maps $x\mapsto x\pi(x)$ and $x\mapsto q(\pi(x))$,
which play an important role in the structure of
(the non-abelian rank one residue of) the quadrangles of type $E_6, E_7$ and $E_8$, also play an important role in the \FTS. 
\begin{theorem}\label{th: FTS}
Let  $(K,L,q,1,X,\cdot,h,\theta)$ be a quadrangular algebra over a field $K$ with $\Char(K)\neq2,3$,
and let $\pi$ be as in Definition \ref{def:quad}.
Then $X$ is a Freudenthal triple system with triple product \[xyz:= \tfrac{1}{2}( x(h(y,z)+h(z,y))+y(h(x,z)+h(z,x))+z(h(x,y)+h(y,x))) \] and skew symmetric bilinear form $\langle x,y\rangle:=g(x,y)$,
for all $x,y,z\in X$. 
This Freudenthal triple system is simple and not reduced.

Furthermore we have that $xyz$ is the linearization%
\footnote{Of course, we mean that the map from $X \times X \times X$ to $X$ mapping $(x,y,z)$ to $xyz$ is the linearization of the cubic map
from $X$ to $X$ mapping $x$ to $x\pi(x)$, but our slight abuse of language should not cause any confusion.
Note that we use the convention that the linearization of a homogeneous map $A$ of degree $n$ is the symmetric $n$-linear map $B(x_1,\dots,x_n)$
such that $B(x,\dots,x) = n! \cdot A(x)$.  }
of $x\pi(x)$ and $\langle x, yzw\rangle$ is the linearization of $-\tfrac{1}{2}q(\pi(x))$. 

In particular, $xxx=6x\pi(x)$ and $\langle x,xxx\rangle=-12q(\pi(x))$.
\end{theorem}
\begin{proof}
It is clear that the triple product is symmetric and trilinear. 
It follows from \eqref{eign:quadr} that $g$ is skew symmetric and bilinear, and that $\langle x,yzw\rangle$ is linear in its four variables. 
Since $\pi(x)=\tfrac{1}{2}h(x,x)$ we have that $xxx=6x\pi(x)$, so $xyz$ is the linearization of $x\pi(x)$. 
To prove the first axiom of Definition \ref{def: FTS} we expand $\langle x,yzw\rangle$, and we find
\begin{align*}
g( x,yzw)&=\tfrac{1}{2}f(h(x,yzw),1)\\
&=\tfrac{1}{4}\big(f(h(x,w(h(y,z)+h(z,y))),1)+f(h(x,y(h(w,z)+h(z,w))),1)\\
&\hspace*{12ex}+f(h(x,z(h(w,y)+h(y,w))),1)\big)\\
&=\tfrac{1}{4}\big(f(h(x,w),\overline{h(y,z)+h(z,y)})+f(h(x,y),\overline{h(w,z)+h(z,w)})\\
&\hspace*{12ex}+f(h(x,z),\overline{h(w,y)+h(y,w)})\big)\\
&=-\tfrac{1}{4}\big(f(h(x,w),h(y,z)+h(z,y))+f(h(x,y),h(w,z)+h(z,w))\\
&\hspace*{12ex}+f(h(x,z),h(w,y)+h(y,w))\big).
\end{align*}
Therefore $\langle x,yzw\rangle$ is indeed symmetric and linear in its four variables. 
When we put $x=y=z=w$, this expression equals $-12q(\pi(x))$. Thus it is the linearization of $-\tfrac{1}{2}q(\pi(x))$.
This map is non-zero since both $q$ and $\pi$ are anisotropic.

In order to  establish the second axiom, we show that
\begin{equation}\label{eq:FTS2}
	(x\pi(x))xy=\tfrac{1}{2}\big(f(h(y,x),1) x\pi(x)+f(h(y,x\pi(x)),1) x\big).
\end{equation}
We expand the left side of this identity, and we get
\begin{multline*}
	(x\pi(x))xy = \tfrac{1}{2}\Bigl( x\pi(x) \bigl( h(x,y)+h(y,x) \bigr) + x \bigl( h(x\pi(x),y)+h(y,x\pi(x)) \bigr) \Bigr) \\
	+ y \bigl( h(x,x\pi(x))+h(x\pi(x),x) \bigr).
\end{multline*}
It follows from Theorem~\ref{eign:quadr}(iv) that the third term is zero.
To reduce the two other terms we use
\[h(y,z)+h(z,y)=h(y,z)-\overline{h(y,z)}=2h(y,z)-f(h(y,z),1) \, 1 ,\]
and we get
\begin{align*}
(x\pi(x))xy
&=\tfrac{1}{2} \Bigl( 2 x\pi(x)h(x,y)-x\pi(x)f(h(x,y),1) \\
	&\hspace*{12ex} + 2xh(x\pi(x),y)-xf(h(x\pi(x),y),1) \Bigr) \\
&=x\bigl( \theta(x,h(x,y))+h(x\pi(x),y) \bigr) \\
	&\hspace*{12ex} + \tfrac{1}{2}\bigl(x\pi(x)f(h(y,x),1)+xf(h(y,x\pi(x)),1)\bigr) ,
\end{align*}
where we have used (D1).
It follows from Theorem~\ref{eign:quadr}(iv) that the first term is zero, establishing~\eqref{eq:FTS2}.

As $\langle x,xxx\rangle=-12 q(\pi(x))$ is anisotropic it follows from Lemma \ref{nonred and simple} that the Freudenthal triple system we obtained is simple and not reduced. 
\end{proof}

\begin{remark}
One should be careful not to confuse between the notation for the triple product $xyz$ for $x,y,z\in X$ and the map $X\times L\rightarrow X$, defined in  Definition \ref{def:quad}, which have also denoted by juxtaposition.
However, as there is no multiplication defined on $X$, we will never write $xy$ with $x,y \in X$, so our notation is always unambiguous.

On the other hand, for  $x \in X$ and $v,w\in L$ the term $xvw$ could be interpreted in two ways.
However without brackets this will always denote the triple product, whereas with brackets $(xv)w$ this denotes applying the $X\times L\rightarrow X$ map two successive times.
\end{remark}

In the next theorem we show that the triple product behaves well with respect to the $C(q,1)$-module structure on $X$ (see Remark~\ref{rem:module}).
\begin{theorem}\label{th: module FTS}For $x,y,z\in X$ and $v\in L\setminus\{0\}$ we have that
\[(xyz)v=\frac{(xv)(yv)(zv)}{q(v)}.\]
\end{theorem}
\begin{proof}
It is enough to show that this identity holds for $x=y=z$, since the general result then follows by linearizing.
Thus we have to show that
\[ (x\pi(x))v=\frac{(xv)\pi(xv)}{q(v)}.\]
This follows from \cite[Theorem 3.18]{W67}, since $(x\pi(x))v=x\theta(x,v)$; the map $\phi$ occuring in that formula is identically zero for fields of characteristic not~$2$.
(In {\em loc.\@ cit.\@}, only quadrangular algebras of type $E_6, E_7$ and $E_8$ are considered, but this proof is also valid for pseudo-quadratic spaces.)
%
%
\begin{comment}
Applying Theorem~\ref{eign:quadr}(v) we get
\[(xv)\pi(xv)= q(v) (xv)\pi(x)^\si-f(1,v)(xv)\theta(x,v)^\si+f(\theta(x,v),1)(xv)v^\si.\]
We first rewrite $(xv)\theta(x,v)^\si$ as
\begin{align*}
	(xv)\theta(x,v)^\si
	&= -(x\theta(x,v))v^\si + f(\theta(x,v),v)x \\
	&= -\bigl( (x\pi(x)) v \bigr) v^\si \\
	&= -q(v) x\pi(x),
\end{align*}
so
\begin{align*}
\frac{(xv)\pi(xv)}{q(v)}&=(xv)\pi(x)^\si-\frac{f(1,v)}{q(v)}(xv)\theta(x,v)^\si+f(\theta(x,v),1)x\\
&= -(x\pi(x))v^\si+f(\pi(x),v)x + f(1,v) x \pi(x) + f(\theta(x,v),1)x \\
&= (x\pi(x))(f(1,v)1-v^\si)+(f(\pi(x),v)+f(\theta(x,v),1))x\\
&= (x\pi(x))v.
\tag*{\qedhere}
\end{align*}\end{comment}
\end{proof}

%%%%%%%%%%%%%%%%%%%%%%%%%%%%%%%%%%%%%%%%%%%%%%%%%%%%%%%%%%%%%%%%%%%%%%%%%%%%%%%%%%%%%%%%%%%%%%%%%%%%%%%%%%%%%%%%%%
%                                                                                                                %
%  SECTION : STRUCTURABLE ALGEBRAS ON AN ARBITRARY QUADRANGULAR ALGEBRA                                          %
%                                                                                                                %
%%%%%%%%%%%%%%%%%%%%%%%%%%%%%%%%%%%%%%%%%%%%%%%%%%%%%%%%%%%%%%%%%%%%%%%%%%%%%%%%%%%%%%%%%%%%%%%%%%%%%%%%%%%%%%%%%%
\section{Structurable algebras on arbitrary quadrangular algebras}\label{se:strQA}

In this main section of the paper, we will show that every quadrangular algebra in characteristic not $2$ or $3$ gives rise
to a family of isotopic structurable algebras, in such a way that
certain concepts from the theory of the quadrangular algebras and from the theory of structurable algebras coincide.
\begin{theorem}\label{th:Q-main}
Let $(K,L,q,1,X,\cdot,h,\theta)$ be a quandrangular algebra with $\Char(K) \neq 2,3$.
Then there exists a family of pairwise isotopic structurable algebras on $X$, such that each algebra $\A$ in this family satisfies the following properties:
\begin{compactenum}[\rm (i)]
\item $\A$ has skew-dimension one,
\item $\A$ is a division algebra,
\item there is a skew-symmetric element $s_0 \in \Ss$ such that $U_x (s_0 x) = 3x\pi(x)$ for all $x\in \A$,
\item the conjugate norm of $\A$ is a scalar multiple of $q\circ \pi$,
\item the conjugate inverse in $\A$ behaves as the inverse in the Moufang quadrangle; see section~\textup{\ref{par: inverses}} below.
\end{compactenum}
\end{theorem}
Theorem \ref{th:Q-main} is a consequence of Theorems \ref{th: FTS}, \ref{th: gar} and \ref{th: struct expliciet} below.
\subsection{The main construction}

In order to define a structurable algebra on $X$, we make use of the Freudenthal triple system that we have described in the previous section.
\begin{theorem}\label{th: gar}
Let $(V,t,b)$ be a simple Freudenthal triple system. There exists a structurable algebra $(\A,\bar{\ })$ of skew-dimension one, such that $(V,t,b)$ is isometric to $(\A,\bar{\ })$, considered as a \FTS\ as in Theorem~\textup{\ref{th:FTS skew dim}}.
\end{theorem}
\begin{proof}
We can apply the construction in \cite[Lemma 4.15]{G}.
In {\em loc.\@ cit.\@}, a \FTS\ is defined to be of dimension $56$;
it is however easily verified that this construction can be carried out for simple Freudenthal triple systems of arbitrary dimension.
\end{proof}

We want to describe the structurable algebra, obtained by combining Theorems \ref{th: FTS} and \ref{th: gar}, in a more detailed way than in \cite[Lemma 4.15]{G}.
In order to do this we have to make the construction much more explicit. 
The structurable algebra constructed in Theorem \ref{th: gar} is obtained in three steps:
\begin{compactenum}[{\em Step} 1.]
\item We tensor the simple Freudenthal triple system $X$ with a quadratic field extension $\Delta$ such that it becomes reduced. 
\item We apply the proof of \cite[Theorem 5.1]{F} to construct a structurable matrix algebra that is isometric to $X \otimes_K \Delta$.
\item We use the methods from \cite[Lemma 4.15]{G} to apply Galois descent and find a structurable algebra that is isometric to $X$. 
\end{compactenum}

%%%%
%%%%  MAIN CONSTRUCTION STARTS HERE
%%%%
\begin{construction}\label{con:main}
Let  $\Omega = (K,L,q,1,X,\cdot,h,\theta)$ be a quadrangular algebra with char$(K)\neq2,3$,
and consider $X$ as a simple non-reduced \FTS\ as in Theorem \ref{th: FTS}.

\step{1}{Extending scalars to make $X$ reduced}

\noindent
To reduce $X$ we use Lemma \ref{lem: sre in F}. For all $x\in X$, we have $\langle x,xxx\rangle=12(-q(\pi(x)))$.
Since $X$ is not reduced, $-q(\pi(x))$ is never a square in $K$. 

We fix an arbitrary $a\in X^*$ and define $\delta:=\sqrt{-q(\pi(a))}$ in the algebraic closure of $K$, so that $\Delta=K(\delta)$
is a quadratic field extension of $K$; let $\iota$ be the non-trivial element of $\Gal(\Delta/K)$.

We now linearly extend the trilinear product and the bilinear form on $X$ to $X \otimes_K \Delta$. This makes $X\otimes_K \Delta$ into a \FTS.
By our choice of $\Delta$, the \FTS\ $X \otimes_K \Delta$ is reduced. By Lemma~\ref{lem: sre in F} and Theorem~\ref{th: FTS},
\[ u'_1=\tfrac{1}{2} \Bigl( a+\frac{a\pi(a)}{\delta} \Bigr), \quad u'_2=\tfrac{1}{2\delta}\Bigl(-a+\frac{a\pi(a)}{\delta} \Bigr) \]
form a supplementary pair of strictly regular elements.
 
\step{2}{Construction of a structurable matrix algebra isometric to $X \otimes_K \Delta$}

\noindent
We point out that if we say that a structurable matrix algebra $M(J,\eta)$ is isometric to $\XtensorE$, we mean that the \FTS\  $M(J,\eta)$, defined by the formulas for $\langle.,.\rangle$ and $\nu$ in Example~\ref{ex: skew dim}, is isometric to the \FTS\ $\XtensorE$. 

In order to construct a structurable matrix algebra that is isometric to $\XtensorE$, we have to construct a Jordan algebra over $\Delta$. 
We proceed as in~\cite{F}, but we slightly modify the construction which is presented there.
We only give the necessary ingredients, referring the reader to {\em loc.\@ cit.\@} for more details.

For $\epsilon \in \{ 1, -1 \}$, we let
\[ M_\epsilon := \{ x \in \XtensorE \mid u'_1 u'_2 x = \epsilon x \}. \]
As in {\em loc.\@ cit.\@}, we will define a Jordan algebra on the vector space $M_1$.
This Jordan algebra will be constructed either from a quadratic form or from an admissible cubic form. 

If the expression $g(u'_1,y\pi(y))$ is identically zero for $y\in M_1$, then there is a quadratic form $Q$ on $M_1$ making $M_1$ into a Jordan algebra;
in this case we define $N=0$ and $\lambda = 1$.
 
On the other hand, if there exists an $e\in M_1$ such that $g(u'_1,e\pi(e))\neq0$, then
\[ N(x):=\frac{g( u'_1,x\pi(x))}{g(u'_1,e\pi(e))}\]
is an admissible cubic form on $M_1$ with base point $e$, making $M_1$ into a Jordan algebra,
and we let
\[ \lambda := \tfrac{1}{2}g(u'_1,e\pi(e))\in \Delta. \]

It is shown in {\em loc.\@ cit.\@} that in both cases, $\XtensorE$ is isometric to the structurable matrix algebra $M(M_1,\lambda)$
(as defined in Example~\ref{ex: skew dim}).
However, we prefer to slightly modify the construction so that $\XtensorE$ becomes isometric to $M(M_1,1)$.
One obvious way to do this is to redefine the pair of strictly regular elements as $u_1 = \lambda^{-1} u'_1$ and $u_2 = \lambda u'_2$, so that

\[ u_1=\frac{1}{2\la}\Bigl(a+\frac{a\pi(a)}{\delta}\Bigr), \quad u_2=\frac{\la}{2\delta}\Bigl(-a+\frac{a\pi(a)}{\delta}\Bigr);\]
then $\XtensorE$ will indeed be isometric to $M(M_1,1)$. 
Note that the spaces $M_\epsilon$ are unchanged by replacing $u'_1$ and $u'_2$ by $u_1$ and $u_2$, respectively.
 
In {\em loc.\@ cit.\@} it is shown that $\XtensorE=\Delta u_1\oplus \Delta u_2\oplus M_1\oplus M_{-1}$, and that there exists an isomorphism
$t \colon M_{1}\rightarrow M_{-1}$. 
This allows us to explicitly write down the isometry $\psi$ between $\XtensorE$ and $M(M_1,1)$:
\begin{align*}
\psi:\Delta u_1\oplus \Delta u_2\oplus M_1\oplus M_{-1} &\rightarrow \begin{pmatrix}\Delta&M_1\\M_1&\Delta\end{pmatrix} \colon \\
d_1 u_1 +d_2 u_2 +j_1+t(j_2) &\mapsto \begin{pmatrix}d_1&j_1\\j_2&d_2\end{pmatrix},
\end{align*}
for all $d_1, d_2\in \Delta$ and all $j_1, j_2\in M_1$.
So we obtain a structurable algebra $M(M_1,1)$ that is defined over $\Delta$ and isometric to $\XtensorE$.

\step{3}{Galois descent}

\noindent
Our next step is to apply Galois descent to obtain a structurable algebra over $K$ isometric to $X$.
We follow the ideas of \cite[Lemma 4.15]{G}, but we use a more explicit approach in order to obtain exact formulas. 

Let $\widetilde{\eta}$ be the extension of $\iota$ to $\XtensorE$ given by
\[ \widetilde{\eta}(x\otimes d):=x\otimes\iota(d). \]
Since the fixed point set of $\widetilde{\eta}$ in $\XtensorE$ is $X$, we determine how this map acts on $M(M_1,1)$. 
As $\widetilde{\eta}(xyz)=\widetilde{\eta}(x)\widetilde{\eta}(y)\widetilde{\eta}(z)$, the map $\widetilde{\eta}$ is an isometry of the \FTS.
We have $\widetilde{\eta}(u_1)=\frac{-\delta}{N(\lambda)} u_2$, and it follows from $\widetilde{\eta}(u_1u_2x)=-u_1u_2\widetilde{\eta}(x)$ that
$x\in M_{\pm1}$ if and only if $\widetilde{\eta}(x)\in M_{\mp1}$.

The explicit formula for $\widetilde\eta$ is given by
\begin{align*}
&\widetilde{\eta}(d_1 u_1 +d_2 u_2 +j_1+t(j_2)) \\
&\hspace*{12ex} =\iota(d_1)\frac{-\delta}{N(\lambda)} u_2+\iota(d_2) \frac{N(\lambda)}{\delta} u_1 +\widetilde{\eta}(j_1)+\widetilde{\eta}(t(j_2))\\
&\hspace*{12ex} =\iota(d_2) \frac{N(\lambda)}{\delta} u_1+\iota(d_1)\frac{-\delta}{N(\lambda)} u_2  +\widetilde{\eta}(t(j_2))+t(t^{-1}(\widetilde{\eta}(j_1))),
\end{align*}
for all $d_1, d_2 \in \Delta$ and all $j_1, j_2 \in M_1$.
Since $ \widetilde{\eta}(t(j_2)), t^{-1}(\widetilde{\eta}(j_1))\in M_1$, we can translate this into matrix notation using $\psi$, and we get
\[
\renewcommand{\arraystretch}{1.3}
   \widetilde{\eta}: \begin{pmatrix}d_1&j_1\\j_2&d_2\end{pmatrix} \mapsto
   \begin{pmatrix} \frac{N(\lambda)}{\delta}\iota(d_2) &\widetilde{\eta}(t(j_2))\\t^{-1}(\widetilde{\eta}(j_1))&\frac{-\delta}{N(\lambda)}\iota(d_1)\end{pmatrix}.
\renewcommand{\arraystretch}{1} \]

We denote the \FTS\ on $\A:=M(M_1, 1)$ from Example~\ref{ex: skew dim} by $(\A,b,t)$; it follows that $\widetilde{\eta}$ is an isometry of $(\A,b,t)$.
It is important to note, however, that $\widetilde{\eta}$ is in general {\em not} an algebra automorphism of $\A$,
and the fixed points of $\widetilde{\eta}$ in $\A$ do {\em not} form a structurable algebra. 

Following \cite{G}, we consider the structurable algebra $\A':=M(M_1, \frac{\delta}{N(\lambda)})$; denote
the corresponding \FTS\ by $(\A',b',t')$.
We now modify this \FTS\ once more.
Let
\[ s_0'=\frac{N(\lambda)}{\delta}\begin{pmatrix}
1&0\\ 
0&-1
\end{pmatrix} , \]
and consider the \FTS\ associated to $\A'$ with respect to $s_0'$ as in Theorem \ref{th:FTS skew dim}.
Then we obtain the \FTS\ $(\A',b'', t'')$,
where
\[ b''=\frac{N(\lambda)}{\delta}b' \quad \text{and} \quad t''= \frac{N(\lambda)}{\delta}t' \,. \]
The map
\[\widetilde{f}:\A\rightarrow\A':\begin{pmatrix}d_1&j_1\\j_2&d_2\end{pmatrix}\mapsto\begin{pmatrix}\frac{\delta}{N(\lambda)}d_1&j_1\\j_2&d_2\end{pmatrix}\]
is an isometry from $(\A,b,t)$ to $(\A',b'',t'')$.
Now consider the map $\widetilde{\pi}:=\widetilde{f}\widetilde{\eta} \widetilde{f}^{-1}:\A'\rightarrow \A'$,
which is explicitly given by
\[\widetilde{\pi}: \begin{pmatrix}d_1&j_1\\j_2&d_2\end{pmatrix}\mapsto \begin{pmatrix} \iota(d_2) &\widetilde{\eta}(t(j_2))\\t^{-1}(\widetilde{\eta}(j_1))&\iota(d_1)\end{pmatrix}.\] 
It is now obvious that $\widetilde{\pi}$ is an isometry of $(\A',b'',t'')$.
Using some properties of norm similarities of Jordan algebras, one can show that $\widetilde{\pi}$ is, in fact, an algebra automorphism of $\A'$.
 
It follows that $\A'^{\widetilde{\pi}}$, the fixed points of $\A'$ under $\widetilde{\pi}$, is a structurable algebra.
Considered as \FTS s, $\A'^{\widetilde{\pi}}$ and $\A^{\widetilde{\eta}}$ are isometric. Since $\A^{\widetilde{\eta}}$ is in turn isometric to $X$, the map
\[ \tau := \widetilde{f}\circ\psi:X\rightarrow \A'^{\widetilde{\pi}} \]
is an isometry.

We now use this isometry to make $X$ into a structurable algebra isomorphic to $\A'^{\widetilde{\pi}}$, by defining the following multiplication
and involution:
\[ x\star y := \tau^{-1}(\tau(x)\tau(y)) \quad \text{and} \quad \overline{x} := \tau^{-1}(\overline{\tau(x)}) \]
for all $x,y\in X$,
where the multiplication and involution in the right hand sides are as in Example~\ref{ex: skew dim} applied on $\A'$.

We will denote this structurable algebra by
\[ X = X(\Omega, a, \lambda) \,, \]
where $\Omega$ is the quadrangular algebra we started from, and where $a \in X^*$ and $\lambda \in \Delta$
are as in Step 1 and Step 2, respectively.

\hfill {\tt [End of Construction~\ref{con:main}]}
\end{construction}
%%%%
%%%%  MAIN CONSTRUCTION ENDS HERE
%%%%

We can now explicitly write down the structurable algebra $X$ in terms of the original quadrangular algebra.
\begin{theorem}\label{th: struct expliciet}
Let $X = X(\Omega, a, \lambda)$ be as above.
Let
\begin{align*}
{\bf1} &:= \frac{1}{2\delta} \biggl( \la^\si \Bigl( a+\frac{a\pi(a)}{\delta} \Bigr)
	+ \la \Bigl( -a+\frac{a\pi(a)}{\delta} \Bigr) \biggr), \\
s_0 &:= \frac{N(\la)}{2\delta^2} \biggl( \la^\si \Bigl( a+\frac{a\pi(a)}{\delta} \Bigr)
	- \la \Bigl( -a+\frac{a\pi(a)}{\delta} \Bigr) \biggr), \\
\mu &:= \frac{N(\la)^2}{\delta^2}.
\end{align*}
Then $X$ is a structurable algebra with zero element $0 \in X$ and identity element ${\bf 1} \in X$;
$X$ has skew-dimension one, and the subspace $\mathcal{S}$ of skew\nobreakdash-\hspace{0pt}symmetric elements is generated by $s_0$,
with $s_0^2 = \mu$.
The subspace $\mathcal{H}$ of symmetric elements is 
\[\mathcal{H}=\{x\in X\mid \overline{x}=x\}=\{k{\bf 1}+j+\widetilde{\eta}(j)\mid k\in K, j\in M_1\}.\]

Moreover, for all $x,y,z\in X$, we have
\begin{align*}
V_{x,\,s_0 \star y}\,z &= \tfrac{1}{2} \Bigl( xh(y,z)+yh(x,z)+zh(y,x) \Bigr) , \\
(x \star \overline{y}-y \star \overline{x}) \star s_0 &= g(x,y)\bf{1} , \\
\nu(x) &= -\frac{\delta^2}{N(\la)^2} \, q(\pi(x)) ,
\end{align*}
where $\nu$ is the conjugate norm of $X$ (see Remark~\textup{\ref{rem:norm}(ii)}).
If we make other choices for $a' \in X^*$ and $\lambda' \in \Delta$, then the structurable algebras
$X(\Omega, a, \lambda)$ and $X(\Omega, a', \lambda')$ are isotopic.
\end{theorem}

\begin{proof}
By definition, the isometry $\tau$ is an isomorphism from the structurable algebra $X$ to the structurable algebra $\A'^{\widetilde{\pi}}$,
which is known to be of skew-dimension one.
In particular, the zero element and the identity element of $X$ are equal to $0 = \tau^{-1}(0)$ and
${\bf1} := \tau^{-1}\left( \begin{smallmatrix}1&0\\0&1\end{smallmatrix} \right)$, respectively.
Moreover, $X$ has skew-dimension one, and the set $\mathcal{S}$ of skew-symmetric elements is generated by 
$s_0 := \tau^{-1}(s'_0)$.
We now perform some explicit calculations. %, we first determine the map $\tau^{-1}$:

Notice that all elements in $\A'^{\widetilde{\pi}}$ are of the form
\[ \begin{pmatrix}d&j\\t^{-1}(\widetilde{\eta}(j))&\iota(d)\end{pmatrix} \]
for some $d\in \Delta$ and $j\in M_1$.
We compute $\tau^{-1} := \psi^{-1}\circ \widetilde{f}^{-1} \colon \A'^{\widetilde{\pi}} \rightarrow X$:
\begin{align*}
\begin{pmatrix}d&j\\t^{-1}(\widetilde{\eta}(j))&\iota(d)\end{pmatrix}%\in \A'^{\widetilde{\pi}}	
&\xmapsto{ \widetilde{f}^{-1}}\begin{pmatrix}\frac{N(\la)}{\delta}d&j\\t^{-1}(\widetilde{\eta}(j))&\iota(d)\end{pmatrix}\in\A^{\widetilde{\eta}} \\[.5ex]
&\xmapsto{ \psi^{-1}}\frac{N(\la)}{\delta}du_1+\iota(d)u_2+j+\widetilde{\eta}(j)\in X.
\end{align*}

%One can check that $\frac{N(\la)}{\delta}du_1+\iota(d)u_2+j+\widetilde{\eta}(j)\in \XtensorE$ is indeed fixed by $\widetilde{\eta}$, and thus is in $X$.
Now we can determine
\begin{align*}
{\bf 1} &= \tau^{-1}\begin{pmatrix}1&0\\0&1\end{pmatrix}=\frac{N(\lambda)}{\delta}u_1+u_2 \\
	&= \frac{1}{2\delta}\biggl(\la^\si \Bigl(x+\frac{x\pi(x)}{\delta}\Bigr)+\la\Bigl(-x+\frac{x\pi(x)}{\delta}\Bigr)\biggl), \\[1ex]
s_0 &= \tau^{-1}\begin{pmatrix}\frac{N(\lambda)}{\delta}&0\\ 0&-\frac{N(\lambda)}{\delta}\end{pmatrix}
	= \frac{N(\lambda)^2}{\delta^2}u_1-\frac{N(\lambda)}{\delta}u_2 \\
	&= \frac{N(\la)}{2\delta^2} \biggl( \la^\si \Bigl( x+\frac{x\pi(x)}{\delta} \Bigr)
		- \la \Bigl( -x+\frac{x\pi(x)}{\delta} \Bigr) \biggr),\\[1ex]
s_0\star s_0 &= \tau^{-1}(s'_0s'_0)=\tau^{-1}\left(\frac{N(\lambda)^2}{\delta^2} \begin{pmatrix}1&0\\ 0&1\end{pmatrix}\right)
	= \frac{N(\lambda)^2}{\delta^2}{\bf 1}. 
\end{align*}
We determine how the involution of the structurable algebra $\A'$ acts on $X$.
We have $\overline{x}=\tau^{-1}(\overline{\tau(x)})$, therefore
\begin{multline*}
\overline{\frac{N(\la)}{\delta}du_1+\iota(d)u_2+j+\widetilde{\eta}(j)} = \tau^{-1}\left(\overline{
	\begin{pmatrix}d&j \\
	t^{-1}(\widetilde{\eta}(j))&\iota(d)\end{pmatrix}}\right) \\
= \tau^{-1}\left(
	\begin{pmatrix}\iota(d)&j \\
	t^{-1}(\widetilde{\eta}(j))&d\end{pmatrix}
	\right) = \frac{N(\la)}{\delta}\iota(d)u_1+du_2+j+\widetilde{\eta}(j).
\end{multline*}
Since ${\bf 1} =\frac{N(\lambda)}{\delta}u_1+u_2$, it follows that each element in the $K$-vector subspace $\{k{\bf 1}+j+\widetilde{\eta}(j)\mid k\in K, j\in M_1\}$ is fixed by the involution.
Since
\[ \{k{\bf 1}+j+\widetilde{\eta}(j)\mid k\in K, j\in M_1\}\oplus Ks_0=X, \]
we conclude that $\mathcal{H}=\{k{\bf 1}+j+\widetilde{\eta}(j)\mid k\in K, j\in M_1\}$. 
 
By the definition of $\star$ and $x \mapsto \overline{x}$, we have
\[ V_{x,\,s_0 \star y}\,z = \tau^{-1} \bigl( V_{\tau(x),\,s_0'\tau(y)}\,\tau(z) \bigr) . \]
It follows from Theorem~\ref{th:FTS skew dim} that
\begin{multline*}
	V_{x,\,s_0 \star y}\,z = \tau^{-1} \Bigl( \tfrac{1}{2}(\tau(x)\tau(y)\tau(z))+\langle \tau(y),\tau(z)\rangle \tau(x) \\
	+ \langle \tau(y),\tau(x)\rangle \tau(z)+\langle \tau(x),\tau(z)\rangle \tau(y) \Bigr)  .
\end{multline*}
Since $\tau$ is an isometry we have
\begin{align*}
V_{x,\,s_0 \star y}\,z
&= \tfrac{1}{2}(xyz+\langle y,z\rangle x+\langle y,x\rangle z+\langle x,z\rangle y)\\
&= \tfrac{1}{2} \Bigl( \tfrac{1}{2} \bigl( x(h(y,z)+h(z,y))+y(h(x,z)+h(z,x))+z(h(x,y)+h(y,x)) \bigr) \\
	&\hspace*{12ex} + g(y,z) x+ g(y,x) z+g(x,z) y \Bigr) \\
&= \tfrac{1}{2}( xh(y,z)+yh(x,z)+zh(y,x)),
\end{align*}
where the last step follows from
\begin{align*}
	h(y,z)+h(z,y) &= h(y,z)-\overline{h(y,z)} \\
	&= 2h(y,z)-f(h(y,z),1)1=2 \bigl( h(y,z)-g(y,z)1 \bigr).
\end{align*}
Again it follows from Theorem~\ref{th:FTS skew dim} and Theorem~\ref{th: FTS} that
\begin{align*}
(x \star \overline{y}-y \star \overline{x}) \star s_0 &= \tau^{-1} \Bigl( \bigl( \tau(x)\overline{\tau(y)}-\tau(y)\overline{\tau(x)} \bigr) s'_0 \Bigr) \\
&= \tau^{-1} \bigr( \langle \tau(x),\tau(y)\rangle 1 \bigr) \\
&= \langle x,y\rangle {\bf1} = g(x,y) {\bf1}; \\[1ex]
\nu(x) &= \frac{1}{12\mu}\langle x,xxx\rangle=-\frac{\delta^2}{N(\la)^2}q(\pi(x)).
\end{align*}
Finally, if we make other choices for $a' \in X^*$ and $\lambda' \in \Delta$, then the structurable algebras
$X(\Omega, a, \lambda)$ and $X(\Omega, a', \lambda')$ are, by construction, isometric as \FTS s to the \FTS \ $X$.
It follows from Lemma \ref{lem: similar and isotopic} that they are isotopic.
\end{proof}

\begin{remark}
	In the case of quadrangular algebras of type $E_6$, $E_7$ and $E_8$, we can actually take $\lambda = 1$,
	in which case the formulas of Theorem~\ref{th: struct expliciet} look nicer;
	see Lemma~\ref{le:E1} below.
	We do not know whether we can always take $\lambda = 1$ in the pseudo-quadratic case.
\end{remark}
Two quadrangular algebras are isotopic if and only if they describe the same Moufang quadrangle.
For a precise definition and some properties, we refer to \cite[Chapter 8]{W}.
In the following lemma we observe that when we construct two structurable algebras starting from two isotopic quadrangular algebras,
we end up with isotopic structurable algebras.

\begin{lemma} Let $\Omega = (K,L,q,1,X,\cdot,h,\theta)$ and $\Omega' = (K,L',q',1',X',\cdot',h',\theta')$ be two isotopic quadrangular algebras.
Then the structurable algebras constructed from $X$ and from $X'$ as in the previous corollary are isotopic.
\end{lemma}
\begin{proof}
By Lemma \ref{lem: similar and isotopic} it suffices to show that $X$ and $X'$ have isometric \FTS s, since the \FTS s of the quadrangular algebra are isometric to the \FTS s of the obtained structurable algebras.

If $\Omega$ and $\Omega'$ are isotopic, then $\Omega'$ is isomorphic to the isotope $\Omega_u$ for some $u\in L$;
we denote the corresponding isomorphisms from $L_u$ to $L'$ and from $X_u$ to $X'$ by $\alpha$ and $\psi$, respectively.
We use the following formulas from \cite[Proposition 8.1]{W}:
\begin{align*} 
1' &= \alpha(u) , \\
\theta'(\psi(x),\alpha(v)) &= q(u)^{-1}\theta(x,v) , \\
\psi(x) \cdot' \alpha(v) &= (xv)u^{-1} ,
\end{align*}
for all $x\in X$ and all $v\in L$.
It follows that
\begin{align*}
	\psi(x)\cdot'\pi'(\psi(x)) &= \psi(x)\cdot'\theta'(\psi(x),1') = q(u)^{-1}(x\theta(x,u))u^{-1} \\
	&= q(u)^{-1} \bigl( (x\pi(x))u \bigr) u^{-1} = q(u)^{-1}x\pi(x).
\end{align*}
By linearizing this expression we obtain that the \FTS s are similar with isometry $\psi$ and multiplier $q(u)^{-1}$.
\end{proof}
\begin{remark}
	We do not know whether the converse also holds, in other words, whether the fact that the structurable algebras are isotopic implies
	that the quadrangular algebras are isotopic.
\end{remark}

The compatibility of the Freudenthal triple system with the $C(q,1)$-module structure of $X$ translates
into the following corollary, showing that
the $V$-operators of the structurable algebra in Theorem~\ref{th: struct expliciet} also behave well
with respect to the $C(q,1)$-module structure of $X$.
\begin{corollary}\label{co: V-mod}
For all $x,y,z\in X$ and $v\in L\setminus \{0\}$ we have that
\[(V_{x,\,s_0 \star y}\,z)v=\frac{V_{xv,\,s_0 \star (yv)}\,zv}{q(v)} \,.\]
\end{corollary}
\begin{proof}
From Theorem \ref{th: module FTS} we have that $(xyz)v=\frac{(xv)(yx)(zv)}{q(v)}$, and from \cite[Proposition 4.18]{W} it follows that $g(xv,yv)=g(x,y)q(v)$.
Hence
\begin{align*}
	V_{xv,\,s_0 \star (yv)}\,zv &= \tfrac{1}{2} \Bigl( (xv)(yv)(zv)+\langle (yv),(zv)\rangle xv \\
		&\hspace*{22ex} +\langle (yv),(xv)\rangle zv+\langle (xv),(zv)\rangle yv \Bigr) \\
	&= \frac{q(v)}{2} \Bigl( (xyz)v+\langle y,z\rangle xv+\langle y,x\rangle zv+\langle x,z\rangle yv \Bigr) \\
	&= q(v)(V_{x,\,s_0 \star y}\,z)v.
	\tag*{\qedhere}
\end{align*}
\end{proof}

\subsection{Inverses in quadrangular algebras and in structurable algebras}\label{par: inverses}

In this section we show that there is a relation between the conjugate inverse of an element in a structurable algebra
and a map that behaves like an inverse in a quadrangular algebra.

In Theorem \ref{th: invers in struct} the conjugate inverse of elements in structurable algebras of skew-dimension one is characterized.
We apply this to the structurable algebra we obtained in Theorem~\ref{th: struct expliciet}. 

Since the map $\nu$ is a scalar multiple of  $q\circ \pi$ and  $q\circ \pi$ is anisotropic, each element in $X$ except $0$ has a conjugate inverse,
and hence the structurable algebra $X$ is a division algebra.
For each $u\in X\setminus \{0\}$ the conjugate inverse is given by
\begin{equation}\label{eqn: conj inv}
	\hat{u}=-\frac{1}{3\mu\nu(u)}s_0 \star \{u,s_0 \star u,u\}=s_0\frac{u\pi(u)}{q(\pi(u))}\,.
\end{equation}
In quadrangular algebras, we also encounter a certain expression that behaves like an inverse.
This expression occurs in the calculation of the so-called $\mu$-maps of a Moufang polygon;
these $\mu$-maps are certain automorphisms of order two of the Moufang polygon reflecting an apartment.
We refer to \cite[Chapter 6]{TW} for the precise definition (which is not relevant for our purposes).

In \cite[Chapter 32]{TW}, these $\mu$-maps are calculated explicitly.
We will illustrate with a few examples in which sense the $\mu$-map behaves like an inverse.

\begin{examples}
\begin{compactenum}[(1)]
    \item
In the case of Moufang triangles, the root groups are parametrized by an alternative division algebra $A$; see section~\ref{sss:triangles}.
Then for every $t \in A \setminus \{ 0 \}$, we have
\[\mu(x_1(t)) = x_4(t^{-1}) \, x_1(t) \, x_4(t^{-1}).\]

    \item
In the case of Moufang hexagons, the root group $U_1$ is parametrized by an anisotropic cubic norm structure $J$; see section~\ref{sss:hexagons}.
For each $a \in J \setminus \{ 0 \}$, we define $a^{-1} := a^\sharp / N(a)$, and we have
\[\mu(x_1(a))=x_7(a^{-1}) \, x_1(a) \, x_7(a^{-1}) .\]
    \item
In the case of Moufang quadrangles of quadratic form type, the root group $U_4$ is parametrized by a vector space $V$ equipped with
an anisotropic quadratic form $q$; see section~\ref{sss:quadrangles}.
For each $v \in V \setminus \{ 0 \}$, we have
\[ \mu(x_4(v))=x_0(v / q(v)) \, x_4(v) \, x_0(v / q(v)) . \]
Observe that the element $v / q(v)$ is precisely the inverse of $v$ in the Clifford algebra $C(q)$.

    \item
In our last example, we consider Moufang quadrangles arising from a quadrangular algebra $\Omega = (K,L,q,1,X,\cdot,h,\theta)$.
In this case, $U_1$ is parametrized by $X\times K$, equipped with a group operation
\[ (a,s) \cdot (b,t) = (a+b, s+t+g(b,a)) \]
for all $a,b \in X$ and all $s,t \in K$; see \cite[Chapter 11]{W}.
In the $E_6$, $E_7$ and $ E_8$ case, the corresponding $\mu$-map is calculated in \cite[(32.8)]{TW}, but it is easily verified
that this formula holds for any quadrangular algebra.
For all $x\in X \setminus \{ 0 \}$, we have
\[\mu(x_1(x,0))=x_5 \Bigl( \frac{x\pi(x)}{q(\pi(x))},0 \Bigr) \, x_1(x,0) \, x_5 \Bigl( \frac{x\pi(x)}{q(\pi(x))},0 \Bigr) .\]
(The general formula for $\mu(x_1(x,s))$ is more involved, but similar in style.)
Observe that the expression $\frac{x\pi(x)}{q(\pi(x))}$ is almost equal to the conjugate inverse of $x$; see equation~\eqref{eqn: conj inv}.
The fact that the $s_0$ is missing is explained by the fact that in the expression
\[ V_{u, \hat u} \, y = V_{u,\,s_0 \star \frac{u\pi(u)}{q(\pi(u))}} \, y = y , \]
the $s_0$ is necessary to translate the $V$-operator into the quadrangular algebra context;
see Theorem~\ref{th: struct expliciet}.
\end{compactenum}
\end{examples}

%%%%%%%%%%%%%%%%%%%%%%%%%%%%%%%%%%%%%%%%%%%%%%%%%%%%%%%%%%%%%%%%%%%%%%%%%%%%%%%%%%%%%%%%%%%%%%%%%%%%%%%%%%%%%%%%%%
%                                                                                                                %
%  SECTION : STRUCTURABLE ALGEBRAS ON A PSEUDO-QUADRATIC SPACE                                                   %
%                                                                                                                %
%%%%%%%%%%%%%%%%%%%%%%%%%%%%%%%%%%%%%%%%%%%%%%%%%%%%%%%%%%%%%%%%%%%%%%%%%%%%%%%%%%%%%%%%%%%%%%%%%%%%%%%%%%%%%%%%%%
\section{Structurable algebras on pseudo-quadratic spaces}\label{se:pqQA}\label{par: pseudo-quadratic}

There is a standard procedure to construct a structurable algebra from a hermitian form;
this is precisely case (3) from the classification we mentioned in Remark~\ref{rem:cssa}.
\begin{theorem}[{\cite[Examples 8.iii]{A0}}]\label{th: hermitian form}
Suppose $(\mathcal{E},\bar{\ })$ is a unital associative algebra over a field $K$ with involution $\bar{\ }$.
Let $\mathcal{W}$ be a unital left $\mathcal{E}$-module.
Suppose $h:\mathcal{W}\times\mathcal{W}\rightarrow \mathcal{E}$ is a hermitian form, i.e.
\begin{compactitem}
\item h is bilinear over $K$ and $h(e w_1,w_2)=eh( w_1,w_2)$,
\item $\overline{h(w_1,w_2)}=h(w_2,w_1)$,
\end{compactitem}
for all $e \in \mathcal{E}, w_1,w_2\in \mathcal{W}$.
Then $\mathcal{E}\oplus\mathcal{W}$ is a structurable algebra with the following involution and multiplication:
\begin{align*}
	\overline{e+w} &= \overline{e}+w, \\
	(e_1+w_1)(e_2+w_2) &= (e_1e_2+h(w_2,w_1))+(e_2 w_1+\overline{e_1} w_2),
\end{align*}
for all $e,e_1,e_2 \in \mathcal{E}$ and all $w,w_1,w_2\in \mathcal{W}$.
\end{theorem}

When we start with a pseudo-quadratic space, we have a skew-hermitian form at our disposal;
see Definition~\ref{def:pqs} above.
In characteristic different from two there is a standard procedure to make a skew-hermitian form into a hermitian form. 

We show that when we use this method to construct a structurable algebra on a pseudo-quadratic space
defined over a {\em quadratic pair},
we get an algebra that is isotopic to the family of isotopic algebras constructed in Theorem \ref{th: struct expliciet}.

\begin{definition}[{\cite[Definition 1.12]{W}}]
Let $L$ be a skew field and $\si$ an involution of $L$.
We call $(L,\si)$ a {\em quadratic pair%
\footnote{This notion, taken from \cite[Definition 1.12]{W}, is quite different from the notion of a quadratic pair as defined in the Book of Involutions \cite{KMRT},
and has nothing to do with the notion of a quadratic pair in (finite) group theory either.}},
if either
\begin{compactenum}[(i)]
    \item
	$L/K$ is a separable quadratic field extension and $\si$ is the generator of the Galois group; or
    \item
	$L$ is a quaternion algebra over $K$ and $\si$ is the standard involution.
\end{compactenum}
Define $q(u)=uu^\si$; then $(K,L,q,1)$ is a pointed anisotropic quadratic space.

A pseudo-quadratic space $(L, \si, X, h, \pi)$ where $(L, \si)$ is a quadratic pair,
is called {\em standard} if
\[\pi(xu)= u^\si\pi(x) u\]
for all $x \in X$ and all $u \in L$.
In \cite[Proposition 1.18]{W} it is shown that a standard anisotropic pseudo-quadratic space over a quadratic pair is a quadrangular algebra.
\end{definition}

The following definition is used in the $E_6$, $E_7$ and $E_8$ case in \cite{TW}.
\begin{definition}
Let $(L,\si,X,h,\pi)$ be a pseudo-quadratic space in characteristic not $2$.
We fix an arbitrary element $\xi \in X$; then $\xi L$ becomes a subspace of $X$.
We define the {\em orthogonal complement} of $\xi L$ as
\[ (\xi L)^\perp=\{x\in X\mid g(x,\xi v)=0 \ \text{ for all } v\in L\} . \]
\end{definition}

\begin{lemma}
Let $(L,\si,X,h,\pi)$ be a pseudo-quadratic space in characteristic not $2$.
Then $X= \xi L \oplus (\xi L)^\perp$.
Moreover, for every $x\in X$, we have that $x\in (\xi L)^\perp$ if and only if $h(x, \xi)=0$.
\end{lemma}
\begin{proof}
Our proof is essentially as in \cite[(13.51)]{TW}.
The first assertion follows since $g$ is nondegenerate.
To prove the second assertion let $x \in X$, and observe that for each $v \in L$, we have
\[ g(x, \xi v) = \tfrac{1}{2} f(h(x, \xi v), 1) = \tfrac{1}{2} f(h(x, \xi), v) . \]
The claim now follows since both $f$ and $g$ are nondegenerate.
\end{proof}

The quadratic pair $(L,\si)$ is a unital associative algebra with involution, but $(\xi L)^\perp$ is a right $L$-module equipped with a skew-hermitian form.
In the following lemma, inspired by \cite[(16.18)]{TW}, we redefine the involution on~$L$, the scalar multiplication on $(\xi L)^\perp$ and the skew-hermitian form,
in such a way that we get a module that satisfies the requirements of Theorem \ref{th: hermitian form}.

We embed $L$ in $X$ by considering $\xi L$, which we view as a unital associative algebra with involution by defining
\[(\xi v)(\xi w)=\xi vw\quad \text{and}\quad (\xi v)^\si=\xi v^\si,\]
for all $v,w \in L$.
\begin{lemma}\label{lem: skewherm to herm}
Let $(L,\si,X,h,\pi)$ be a pseudo-quadratic space, and take an element $e\in L$ such that $e^\si=-e$. 
Then the map $v \mapsto \overline{v} = ev^\si e^{-1}$ is an involution of $L$, and $\{s\in L\mid\overline{s}=-s\}=Ke$.

For each $v\in L$ and each $x\in (\xi L)^\perp$, we define $(\xi v)\circ x := x\overline{v}$.
Now let
\[ h' \colon (\xi L)^\perp\times (\xi L)^\perp\rightarrow \xi L \colon (x,y) \mapsto h'(x,y):=\xi eh(x,y) . \]
Then $(\xi L)^\perp$ is a left $\xi L$-module w.r.t.\@ $\circ$, and $h'$ is a hermitian form on $(\xi L)^\perp$ w.r.t.\@ the involution $\bar{\ }$,
satisfying the requirements of Theorem~\textup{\ref{th: hermitian form}}.
\end{lemma}

\begin{proof}
It is clear that the map $v \mapsto \overline{v}$ is an involution.
In order to determine the subspace of skew-symmetric elements $\mathcal{S} = \{s\in L\mid\overline{s}=-s\}$,
we first observe that $\overline{e} = ee^\si e^{-1} = -e$, hence $e \in \mathcal{S}$.
On the other hand, for each $s \in L$, we have
\begin{align*}
\overline{s} = -s &\iff se = - es^\si \\
&\iff se = (se)^\sigma \\
&\iff se \in \Fix_L(\sigma) = K .
\end{align*}
It follows that $\dim_K \mathcal{S} = 1$; since $e \in \mathcal{S}$, we conclude that $\mathcal{S} = Ke$.

Let $v,w \in L$ and $x\in (\xi L)^\perp$, since $h(x \overline{v},\xi)=v h(x,\xi) =0$ we have that $(\xi v)\circ x\in (\xi L)^\perp$. To verify that $(\xi L)^\perp$ is a left $\xi L$-module, it is enough to check the following
\[ \xi v\circ (\xi w\circ x)=(x\overline{w})\overline{v}=x( \overline{w}\ \overline{v})=x(\overline{vw})=(\xi vw)\circ x.\]
Finally, it is easily checked that $h'$ is a hermitian form:
\begin{align*}
&h'(\xi v\circ x, y)=\xi eh(x\overline{v},y)=\xi e(ev^\si e^{-1})^\si h(x,y)=(\xi v) h'(x,y),\\
&\overline{h'(x,y)}=-\overline{h(x,y)}e=-eh(x,y)^\si=h'(y,x),
\end{align*}
for all $x,y \in (\xi L)^\perp$ and all $v \in L$.
\end{proof}

\begin{theorem}\label{th: herm struct alg}
Let $(L,\si,X,h,\pi)$ be a pseudo-quadratic space, and choose an element $e\in L$ such that $e^\si=-e$. 
Since $X =(\xi L)\oplus (\xi L)^\perp$, each element in $X$ can be written in a unique way as $\xi v+x$ for $v\in L$ and $x\in (\xi L)^\perp$.

Then $X$ is a structurable algebra, with involution and multiplication given by
\begin{align}\label{mult herm}
	\overline{\xi v+x} &= \xi(ev^\si e^{-1})+x, \notag\\
	(\xi v+x) \cdot (\xi u+y) &= \xi \bigl( vu+eh(y,x) \bigr) + \bigl( x(eu^\si e^{-1}) + yv \bigr),
\end{align}
for all $u,v\in L$ and all $x,y\in (\xi L)^\perp$.
This structurable algebra has skew-dimension one.
 
When we fix $e = h(\xi ,\xi)$, then this algebra is isotopic to the family of structurable algebras obtained in Theorem~\textup{\ref{th: struct expliciet}}
starting from the pseudo-quadratic space $(L,\si,X,h,\pi)$.

In the corresponding \FTS\  we have $yyy=6y\pi(y)$ for all $y\in X$. 
\end{theorem}
\begin{proof}
We consider $(\xi L)^\perp$ as a left hermitian space over $\xi L$ with involution and hermitian form as in Lemma~\ref{lem: skewherm to herm}.
Writing down the involution and multiplication in Theorem~\ref{th: hermitian form} gives the above formulas.

Since $(\xi L)^\perp$ is invariant under the involution, it follows from Lemma \ref{lem: skewherm to herm} that this structurable algebra has skew-dimension one.
We will denote this structurable algebra, obtained from the hermitian form, by $X$.

On the other hand, let $\widetilde X$ be the structurable algebra obtained from the quadrangular algebra $(L, \sigma, X, h, \pi)$, as in Theorem~\ref{th: struct expliciet}.
We will show that $X$ and $\widetilde X$ are isotopic by showing that their associated \FTS s are similar.

Now take $e=h(\xi,\xi)\in L$; we have $\overline{e}=-e$.
We will determine the trilinear map of the corresponding \FTS\ of $X$ (see Theorem~\ref{th:FTS skew dim}) for $s_0=e$.
Let $y=\xi v+x\in (\xi L)\oplus (\xi L)^\perp$ be arbitrary. Then
\begin{align}
2V_{y,\,s_0y}\,y &= 2(2(y\cdot \overline{s_0\cdot y})y-(y\cdot \overline{y})\cdot (s_0\cdot y))\label{lijn1}\\
&=6 \bigl( \xi(-vev^\si+eh(x,x)e)v+x(e(vev^\si e^{-1}-eh(x,x))\bigr)\label{lijn2}\\
&= 12\bigl(\xi( -\pi(\xi v^\si+xe)v)+x(e\pi(\xi v^\si+xe)e^{-1}) \bigr);\label{lijn3}
\end{align}
equation \eqref{lijn1} follows from the definition of the $V$-operator of a structurable algebra; \eqref{lijn2} follows after a straightforward calculation using the multiplication on $X$ defined by \eqref{mult herm}; and \eqref{lijn3} follows from
\begin{multline*}
vev^\si-eh(x,x)e = vh(\xi,\xi)v^\si-eh(x,x)e=h(\xi v^\si, \xi v^\si)+h(xe,xe)\\
=h(\xi v^\si+xe,\xi v^\si+xe)=2\pi(\xi v^\si+xe),
\end{multline*}
 since $xe\in (\xi L)^\perp.$

It follows from Theorem~\ref{th: struct expliciet} that in $\widetilde X$ the trilinear map of the corresponding \FTS\  is $2V_{\widetilde{y},\,s_0\widetilde{y}}\widetilde{y}=3 \widetilde{y}h(\widetilde{y},\widetilde{y})=6\widetilde{y}\pi(\widetilde{y})$ for all $\widetilde{y} \in \widetilde{X}$.

Let $\psi$ be the vector space isomorphism
\[ \psi \colon \widetilde X \rightarrow X \colon \xi v+x\mapsto \xi v^\si+xe^{-1} ; \]
then by applying~\eqref{lijn3} with $y$ replaced by $\psi(y) = \xi v^\si+xe^{-1}$, we get
\begin{align*}
2V_{\psi(y),\,s_0\psi(y)}\,\psi(y) &= 12 \bigl( -\xi (\pi(y)v^\si)+ x\pi(y)e^{-1} \bigr) \\ 
&=12\psi\bigl(\xi(v\pi(y))+x\pi(y)\bigr) \\
&= 12 \psi\bigl(y\pi(y)\bigr) \\
&= 2 \psi\bigl( 6 y\pi(y)\bigr) .
\end{align*}
The expression $V_{\psi(y),\,s_0\psi(y)}\,\psi(y)$ is exactly the triple product $\psi(y)\psi(y)\psi(y)$ in the \FTS\ of the structurable algebra $X$,
whereas the expression $6y\pi(y)$ is the triple product $yyy$ in the \FTS\  of the structurable algebra $\widetilde X$.

We conclude that $\psi$ is a similarity of \FTS s, and therefore the structurable algebras $X$ and $\widetilde X$ are isotopic (see Lemma~\ref{lem: similar and isotopic}).
\end{proof}

%%%%%%%%%%%%%%%%%%%%%%%%%%%%%%%%%%%%%%%%%%%%%%%%%%%%%%%%%%%%%%%%%%%%%%%%%%%%%%%%%%%%%%%%%%%%%%%%%%%%%%%%%%%%%%%%%%
%                                                                                                                %
%  SECTION : STRUCTURABLE ALGEBRAS ON A QUADRANGULAR ALGEBRA OF TYPE E_6, E_7 AND E_8                            %
%                                                                                                                %
%%%%%%%%%%%%%%%%%%%%%%%%%%%%%%%%%%%%%%%%%%%%%%%%%%%%%%%%%%%%%%%%%%%%%%%%%%%%%%%%%%%%%%%%%%%%%%%%%%%%%%%%%%%%%%%%%%
\section{Structurable algebras on quadrangular algebras of type $E_6$, $E_7$ and $E_8$}\label{se:exQA}\label{par: E6E7E8}

In this section we consider quadrangular algebras of type $E_6, E_7$ and $E_8$.
We investigate the structurable algebra obtained in Theorem \ref{th: struct expliciet} for a ``nice" choice of $a$ and $e$.

In the $E_8$-case, we show that this structurable algebra is a twisted version of the Jordan algebra of a biquaternion algebra over the base field $K$.
This algebra can also be obtained by considering a generalized Cayley--Dickson process, see \cite{CD},
starting from the Jordan algebra of a biquaternion algebra over $K$.
A similar description holds for $E_6$ and $E_7$.
More precisely, we will show the following result.
\begin{theorem}\label{th:E-main}
	Let $\Omega$ be a quadrangular algebra of type $E_6$, $E_7$ or $E_8$ over~$K$, with $\Char(K) \neq 2$,
	and let $X$ be the structurable algebra obtained
	from $\Omega$ as in Theorem~\textup{\ref{th: struct expliciet}} (which is uniquely defined up to isotopy).
	Then $X$ is isotopic to $\CD\bigl(A^+, \Nrd, c \bigr)$ for some division algebra $A$ and some $c \in K$, where
	\begin{compactenum}[\rm (i)]
	    \item
		$A$ is a quaternion algebra $Q$ if $\Omega$ is of type $E_6$;
	    \item
		$A$ is a tensor product $Q \otimes L$ with $Q$ a quaternion algebra and $L$ a quadratic extension,
		if $\Omega$ is of type $E_7$;
	    \item
		$A$ is a biquaternion algebra $Q_1 \otimes Q_2$ if $\Omega$ is of type $E_8$.
	\end{compactenum}
\end{theorem}

We will give more precise statements below; in particular, we will explicitly construct the algebra $A$ and the constant $c$ in each case.

\subsection{Coordinatization}

For an explicit description of the quadrangular algebra of type $E_6$, $E_7$ and $E_8$, we refer to \cite[Chapter 12 and 13]{TW};
for a concise description we refer to the first part of \cite[Chapter 10]{W}. 
Some care is needed, since the map $g$ in \cite{TW} is equal to $-g$ in \cite{W}.
Here we only give a concise overview of the structure of a quadrangular algebra of type $E_6$, $E_7$ or $E_8$.

\begin{definition}\label{def: typeE8}
A quadratic space $(K,L,q)$ with base point is of type $E_6$, $E_7$ or $E_8$ if it is anisotropic and there exists a separable quadratic field extension $E/K$,
with norm denoted by $N$, such that:
\begin{compactitem}
\item[$E_6:$ ] there exist $s_2,s_3\in K^*$ such that \[(K,L,q)\cong (K,E^3,N\perp s_2 N\perp s_3 N);\]
\item[$E_7:$ ] there exist $s_2,s_3,s_4\in K^*$ such that $s_2s_3s_4\notin N(E)$ and \[(K,L,q)\cong (K,E^4,N\perp s_2 N\perp s_3 N\perp s_4 N);\]
\item[$E_8:$ ] there exist $s_2,s_3,s_4,s_5,s_6\in K^*$ such that $-s_2s_3s_4s_5s_6\in N(E)$ and \[(K,L,q)\cong (K,E^6,N\perp s_2 N\perp s_3 N\perp s_4 N\perp s_5 N\perp s_6 N).\]
We always assume that $s_2s_3s_4s_5s_6=-1$, which can be achieved by rescaling the quadratic form if necessary.
We use the convention that $s_{ij}=s_is_j$ and $s_{ijk}=s_is_js_k$ for all $i,j,k \in \{2,\dots,6\}$.
\end{compactitem}
\end{definition}
As we are working in characteristic not $2$, we will always assume that $E=K(\gamma)$ with $\gamma^2\in K$.
 
It is shown in \cite[(12.37)]{TW} that if
\[ (K,E^6,N\perp s_2 N\perp s_3 N\perp s_4 N\perp s_5 N\perp s_6 N) \]
is a quadratic space of type $E_8$, then $(K,E^4,N\perp s_2 N\perp s_3 N\perp s_4 N)$ is a quadratic space of type $E_7$ and $(K,E^3,N\perp s_2 N\perp s_3 N)$ is a quadratic space of type $E_6$.

If $(K,L,q)$ is a quadratic space with base point of type $E_8$, $L$ has dimension 12 over $K$ and there exists a scalar multiplication $E\times L\rightarrow L$ that extends the scalar multiplication $K\times L\rightarrow L$.
We denote a basis of $L$ by $(v_1, \ga v_1,\dots, v_6, \ga v_6)$; with this notation $v_1$ is the base point of $q$.

Let $(K,L,q,1,X,\cdot,h,\theta)$ be a quadrangular algebra of type $E_6$, $E_7$ or $E_8$, then $(K,L,q)$ is a quadratic space of type $E_6$, $E_7$ or $E_8$ with basepoint denoted by $1$.
 
The vector space $X$ has $K$-dimension equal to $8$, $16$ or $32$, respectively; it is a $C(q,1)$-module.
Some of the properties of the maps $\cdot, h, \theta$ and $\pi$ are given in Definition~\ref{def:quad}.
The existence of the vector space $X$ and of the maps $\cdot$, $h$ and $\theta$ is shown in \cite[Chapter 13]{TW}
by giving an explicit ad-hoc construction using the coordinatization of $L$.

Let $\I=\{2,3,4,5,6,23,24,25,26,34,35,36,45,46,56\}$; in the $E_8$-case, an arbitrary element $x \in X$ can be written as
$x = t_1v_1+\sum_{i\in \I}t_iv_i$ with coefficients $t_i\in E$.
The quadrangular algebra in the $E_7$-case is a subspace of the $E_8$-case by taking all the elements in $X$ for which
$t_5=t_6=t_{25}=t_{26}=t_{35}=t_{36}=t_{45}=t_{46}=0$.
The $E_6$-case is the subspace of the $E_7$-case consisting of the elements for which $t_4=t_{24}=t_{34}=t_{56}=0$.

Since the $E_8$-case is the largest and least understood, we focus primarily on that case.
Most of our results in the $E_6$- and $E_7$-case can then be deduced by making the appropriate coefficients zero.

We start by investigating the set
\[ X|_K:=\Bigl\{t_1v_1+\sum_{i\in \I}t_iv_i\mid t_1,\dots, t_{56}\in K\Bigr\} \leqslant X. \]
This is a $16$-dimensional vector space over $K$. 
Let
\[ L|_K=\{t_1v_1+t_2v_2+t_3v_3+t_4v_4+t_5v_5+t_6v_6\mid t_1,\dots, t_{6}\in K\Bigr\}\leqslant L, \]
and denote the restriction of the quadratic form $q$ to $L|_K$ by $q_K \colon L|_K\rightarrow K$.
By construction, see \cite[(13.5) and (13.8)]{TW}, $X|_K$ is isomorphic as a vector space to $C(q_K,1)/ M_K$,
where $M_K$ is the submodule $(v_2v_3v_4v_5v_6-1)C(q_K,1)$ of $C(q_K,1)$. 

Since $v_iv_j=-v_jv_i\in C(q_K,1)$ for $i\neq j \in \{2,\dots,5\}$, the element $v_2v_3v_4v_5v_6$ is in the center of $C(q_K,1)$;
therefore $M_K$ is a two-sided ideal of $C(q_K,1)$. 

\begin{lemma}\label{lem: biquaternion}
Consider $X|_K$ as an associative algebra endowed with the multiplication induced by the Clifford algebra with base point.
Then $X|_K$ is isomorphic, as an algebra, to a biquaternion algebra.
\end{lemma}
\begin{proof}
The multiplication on $X|_K=C(q_K,1)/M_K$ is induced by the multiplication in the Clifford algebra $C(q_K,1)$
by reducing the result modulo $M_K$; see \cite[(13.8)]{TW}.

We define two quaternion algebras over $K$ with the following generators:
\begin{alignat*}{2}
Q_1&:=(-s_2,-s_3)_K &&=\langle \ell,m \mid \ell^2 = -s_2, m^2 = -s_3, \ell m = - m \ell \rangle, \\
Q_2&:=(-s_{46},-s_{56})_K &&=\langle n,r \mid n^2 = -s_{46}, r^2 = -s_{56}, nr = -nr \rangle.
\end{alignat*} 
In order to construct an isomorphism $\psi \colon Q_1\otimes_K Q_2\rightarrow X|_K$, we have to describe
two isomorphisms $\psi_i \colon Q_i \to X|_K$ for $i \in \{1,2\}$, such that the images $\psi_1(Q_1)$ and $\psi_2(Q_2)$ commute elementwise,
and together generate~$X|_K$.
We can achieve this by the choice
\begin{alignat*}{4}
\psi_1(\ell) &= v_2, &\quad \psi_1(m) &= v_3, &\quad \psi_1(\ell m) &= v_{23}, \\
\psi_2(n) &= v_{46}, & \psi_2(r) &= v_{56}, & \psi_2(nr) &= s_6 v_{45}.
\end{alignat*}
Observe that the subspaces $\langle 1, v_2, v_3, v_{23} \rangle$ and $\langle 1, v_{46}, v_{56}, v_{45} \rangle$ do indeed commute elementwise,
and together they generate $X|_K$.
\end{proof}
We can summarize the isomorphism $\psi$ in the following table:
\[
\begin{array}{c|cccc}
\otimes&1&n&r&nr\\
\hline
1&v_1&v_{46}&v_{56}&s_6v_{45}\\
\ell&v_2&-s_{246}v_{35}&s_{256}v_{34}&s_{2456}v_{36}\\
m&v_3&s_{346}v_{25}&-s_{356}v_{24}&-s_{3456}v_{26}\\
\ell m&v_{23}&-s_{2346}v_5&s_{2356}v_4&-v_6
\end{array}\]

\vspace*{0ex}

\begin{remark} \begin{compactenum}[(i)]
\item The construction of the biquaternion algebra depends on the similarity class of $q$ and on the norm splitting for $q$,
and in fact, this biquaternion algebra is {\em not} an invariant of the quadrangular algebra. 
\item The Albert form corresponding to the biquaternion algebra is similar to the $6$-dimensional
quadratic form $\langle 1, s_2, s_3, s_4, s_5, s_6 \rangle$. Since this quadratic form is anisotropic,
we see in particular that the biquaternion algebra $Q_1 \otimes Q_2$ is always a division algebra.
\item In the $E_6$-case $X|_K$ is isomorphic to $Q_1$, and
in the $E_7$-case $X|_K$ is isomorphic to $Q_1\otimes_K K(r)$. 
\end{compactenum}
\end{remark}

The following rather technical lemma will assure that we can choose the structurable algebra $X$ obtained in Theorem~\ref{th: struct expliciet} in a nice way.
\begin{lemma}\label{le:E1}
When applying Theorem~\textup{\ref{th: struct expliciet}} with $a=-v_1/\gamma$, we can always choose $e$ in such a way that $\la=1$.
For those choices we have that  ${\bf 1}=v_1, s_0=-\ga v_1$ and $\mu=\ga^2$. 
\end{lemma}
\begin{proof}
Recall that $\ga\in E\setminus K$ and $\ga^2\in K$. For $a=-v_1/\gamma$ we have that $q(\pi(a))=-\frac{1}{\ga^2}$ and $a\pi(a)=\frac{v_1}{\ga^2}$.
So $\delta=\frac{1}{\ga}$, and hence $\Delta = K(\delta) = E$.
We point out that in $\XtensorE$, the element $1 \otimes \delta$ is not equal to $\frac{1}{\gamma} \otimes 1$.
For instance, we have
\[ \frac{a\pi(a)}{\delta} = \frac{v_1}{\ga^2} \otimes \gamma \neq \frac{v_1}{\ga} \otimes 1 = -a . \]
If we assume that $\la=1$ using the formulas in Theorem  \ref{th: struct expliciet}, we obtain ${\bf 1}=v_1, s_0=-\ga v_1, \mu=\ga^2$.
In order to prove that we can always find an $e$ such that $\la=1$, we will do some explicit calculations.

We have
\[ u'_1=\tfrac{1}{2} \Bigl( -\frac{v_1}{\ga} \otimes 1 +v_1 \otimes \frac{1}{\ga} \Bigr), \quad
	u'_2 = \frac{1}{2} \Bigl( \frac{v_1}{\ga} \otimes \ga + v_1 \otimes 1 \Bigr).\]
We determine explicitly the subspaces $M_1$ and $M_{-1}$.
One can calculate that for all $x=\sum_{I\in \I} t_I v_I\in X$,
\[u'_1u'_2(x \otimes 1) = x/\ga \otimes \ga, \quad u'_1u'_2(x \otimes \ga) = \ga x \otimes 1,\] 
and for all $x=t_1v_1\in X$,
\[u'_1u'_2(x \otimes 1)=2(x/\ga \otimes \ga),\quad  u'_1u'_2(x \otimes \ga)=2(\ga x \otimes 1).\] 
We find that
\begin{equation}\label{eq:M-1}
\begin{aligned}
M_1&=\Bigl\{ \ga x \otimes 1 + x \otimes \ga \Bigm\vert x = \sum_{I\in \I} t_I v_I \text{ with } t_I\in E\Bigr\},\\
M_{-1}&=\Bigl\{ \ga x \otimes 1 - x \otimes \ga \Bigm\vert x = \sum_{I\in \I} t_I v_I \text{ with } t_I\in E\Bigr\}.
\end{aligned}
\end{equation}

Following an idea of Richard Weiss, we introduce the following notation:
let $i,j,k,l,m$ denote five different indices in $\{2,\dots,6\}$, then $\beta_{ijkl}=\pm 1$ is defined by $v_{ij}v_{kl}=\beta_{ijkl}s_is_js_ks_lv_m$.

Next we need an expression for $g(u'_1, e\pi(e))$ for an arbitrary $e\in M_1$.
(Recall that $g$ is now a map from $(\XtensorE) \times (\XtensorE)$ to $\Delta$.)
% In order to formulate this expression we introduce a map $\rho$, as follows.

So let $x=\sum_{2 \leq i \leq 6} t_i v_i + \sum_{2 \leq i<j \leq 6} t_{ij}v_{ij}\in X$ for $t_i, t_{ij}\in E$, and consider the expression 
\[ \rho(x):=\sum_{ij/kl/m}\beta_{ijkl} t_mt_{ij}t_{kl}\in E = \Delta, \] 
where the summation runs over all partitions of $\{2,\dots,6\}$ into two sets of two elements and one set of one element.
% We decompose $\rho$ by setting
% \[ \rho(x)=\rho_1(x)+\rho_2(x)\ga \]
% with $\rho_1(x), \rho_2(x)\in K$. 

Since in the $E_6$-case no such partition with non-zero coefficients exists, $\rho(x)$ is identically $0$ in this case.

Take $e = \ga x \otimes 1 + x \otimes \ga \in M_1$;
then it follows from a lengthy computation%
\footnote{We wrote a computer program in Sage \cite{sage} to perform this computation for us.}
that
% \[g(u'_1, e\pi(e))=16\gamma^4 \bigl( \rho_1(x) \otimes 1 + \rho_2(x) \otimes \ga \bigr)\in K\otimes_K \Delta.\]
\[g(u'_1, e\pi(e))=16\gamma^4 \rho(x) \in \Delta.\]

In the $E_6$-case, this expression is identically $0$, so $\la=1$ by definition. 
% In the $E_7$- and the $E_8$-case, we look for an $e\in M_1$ such that $g(u'_1,e\pi(e))=~2(1\otimes~1)$, so $8\gamma^4(\rho_1(x) \otimes 1 + \rho_2(x) \otimes \ga)$ should be equal to $1$. 
In the $E_7$- and the $E_8$-case, we look for an $e\in M_1$ such that $g(u'_1,e\pi(e)) = 2$, so $8\gamma^4 \rho(x)$ should be equal to $1$. 
This is indeed the case for
\[ x=\tfrac{1}{2\ga^2}(v_2+\gamma v_{34}+\ga v_{56}), \]
since $\beta_{3456}=1$.
\end{proof}

We now show that, with these choices of $a$ and $e$, the structurable algebra $X$ is a ``twisted'' Jordan algebra of a biquaternion algebra,
i.e.\@ when we restrict the coefficients (w.r.t.\@ the standard basis) to $K$, the algebra $X$ is a Jordan algebra of a biquaternion algebra,
but when we allow coefficients of~$E$, we have to apply the non-trivial Galois automorphism $\sigma$ of $E/K$ at various (compatible) places.
\begin{theorem}\label{th: biquat}
Let $\circ$ denote the multiplication of
\[ (X|_K)^+\cong \bigl( (-s_2,-s_3)_K \otimes_K (-s_{46},-s_{56})_K \bigr)^+, \]
the Jordan algebra obtained\!
\footnote{If $A$ is an associative algebra, then $x\circ y:=\tfrac{1}{2}(xy+yx)$ makes $A$ into a Jordan algebra.
This Jordan algebra is denoted by $A^+$.}
from the (associative) biquaternion algebra.

We choose $a=-v_1/\gamma$, and we let $e$ be as in Lemma~\textup{\ref{le:E1}}, such that $\la=1$.
Then the multiplication of the structurable algebra $X$ obtained in Theorem~\textup{\ref{th: struct expliciet}}, which we will denote by $\star$,
is given by
\begin{align*}
Av_1\star Bv_1 &= AB\,v_1\\
Av_1\star Bv_I &= A^\si B\,v_I\\
Av_I\star Bv_1 &= AB\,v_I\\
Av_I\star Bv_I &= A^\si B \, (v_I\circ v_I) = (-s_I) A^\si B\,v_1\\
Av_I\star Bv_J &= A^\si B^\si \, (v_I\circ v_J)
\end{align*}
for all $A, B\in E$ and all $I \neq J\in \I$.
The involution of $X$ is given by
\[ \overline{Av_1}=A^\si v_1,\quad \overline{Av_I}=A v_I \]
for all $A\in E$ and all $I\in \mathcal{I}$.
\end{theorem}

\begin{proof}
By Lemma~\ref{le:E1}, we have ${\bf 1}=v_1$, $s_0=-\ga v_1$ and $\mu=\ga^2$.
We know that $s_0\star s_0=\mu{\bf 1}$, so $\ga v_1\star \ga v_1=\ga^2 v_1$. Since $v_1$ is the identity of $\star $, we have $Av_1\star Bv_1=ABv_1$ for all $A,B\in E$.

Since $X$ has skew-dimension one, it follows from the definition of $s_0$ that \[\Ss=\{x\in X \mid \overline{x}=-x\}=Ks_0=K\ga v_1.\] 
From Theorem~\ref{th: struct expliciet} we have
\[ \mathcal{H}=\{x\in X \mid \overline{x}=x\}=\{k{\bf 1}+j+\widetilde{\eta}(j)\mid k\in K,j\in M_1\}. \]
It follows from equation~\eqref{eq:M-1} that every $j\in M_1$ is of the form $\ga x\otimes 1+ x\otimes \ga$ for some $x\in \bigoplus_{I\in \I}Ev_I$. 
Then $\widetilde{\eta}(j)=\ga x\otimes 1- x\otimes \ga\in M_{-1}.$ It follows that 
\[\{j+\widetilde{\eta}(j)\mid j\in M_1\}=\{2\ga x\otimes 1\mid x\in \bigoplus_{I\in \I}Ev_I\}=\bigoplus_{I\in \I}Ev_I.\]
Therefore
 \[\mathcal{H}=Kv_1\oplus \bigoplus_{I\in \I} Ev_I.\]
It follows that we have for $A\in E$ that
\[\overline{Av_1}=A^\si v_1,\quad \overline{Av_I}=A v_I \text{ for all } I\in \mathcal{I}. \]
For all $x\in \mathcal{H}$ and $y\in X$, we have
\[ V_{x,{\bf 1}}\,y=(x\star \overline{{\bf 1}})\star y+(y\star \overline{{\bf 1}})\star x-(y\star \overline{x})\star {\bf 1}=x\star y. \]
It now follows from Theorem~\ref{th: struct expliciet} that 
\begin{equation}\label{eq:Estar}
	x\star y = V_{x,\,s_0\star \frac{1}{\mu}s_0}\,y = -\frac{1}{2\ga^2} \Bigl( xh(\ga v_1,y)+(\ga v_1)h(x,y)+yh(x,\ga v_1) \Bigr).
\end{equation}
Now we can compute $x\star y$ for all different values that can occur, using the formulas from \cite[(13.6) and (13.19)]{TW}.
Let $i,j,k,l$ be distinct indices in $\{2,3,4,5,6\}$; then one can verify the following multiplication table:
\[
\begin{array}{c|c||c}
x&y&x\star y\\
\hline
Av_i&Bv_1&ABv_i\\
Av_i&Bv_i&-s_i A^\si B v_1\\
Av_i&Bv_k&0\\
Av_i&Bv_{ik}&0\\
Av_i&Bv_{kl}&2A^\si B^\si v_iv_{kl}\\
Av_{ij}&Bv_{1}&AB v_{ij}\\
Av_{ij}&Bv_{ij}&-s_is_j A^\si B v_1\\
Av_{ij}&Bv_{ik}&0\\
Av_{ij}&Bv_{kl}&2A^\si B^\si v_{ij}v_{kl}\\
\end{array}\]
Observe that this multiplication coincides with the Jordan multiplication in $(X|_K)^+$ if $A,B$ are restricted to $K$.

Note that the formula~\eqref{eq:Estar} is not valid for $x=\ga v_1\in \Ss$;
this case is obtained by
\[A v_1\star B v_I=\overline{\overline{Bv_I}\star \overline{A v_1}}=A^\si B v_I. \tag*{\qedhere} \]
\end{proof}
\begin{remark}\label{rem:arason}
The structurable algebra described above, consisting of a twisted Jordan algebra of a biquaternion algebra, is defined up to isotopy by the quadrangular algebra; in particular it is determined by the quadratic form of type $E_6$, $E_7$ or $E_8$.
In the $E_8$ case, there is a strong relation between the Arason invariant of the quadratic form $q$,
and the twisted Jordan algebra of the biquaternion algebra.

The Albert form of the biquaternion algebra in Theorem~\ref{th: biquat} is 
\[q_A=\langle s_2,s_3,s_{23},-s_{46},-s_{56},-s_{45}\rangle;\]
$q_A$ is Witt equivalent to $\llangle -s_2,-s_3\rrangle \perp- \llangle -s_{46},-s_{56}\rrangle$.
In fact, $q_A$ is similar to $\langle 1,s_2,s_3,s_4,s_5,s_6\rangle$ (note that $s_2s_3s_4s_5s_6=-1$):
\begin{align*}
\langle s_2,s_3,s_{23},-s_{46},-s_{56},-s_{45}\rangle&\simeq\langle s_2,s_3,s_{23},s_{235},s_{234},s_{236}\rangle\\
&\simeq s_{23} \langle s_3,s_2,1,s_5,s_4,s_6\rangle\\
&\simeq s_{23} \langle 1,s_2,s_3,s_4,s_5,s_6\rangle.
\end{align*}

Let $I^3K$ be the ideal in the Witt ring of $K$ that is generated by the $3$-fold Pfister forms
%\footnote{Let $a,b,c\in K$, the 3-fold Pfister form is the quadratic form $\langle 1,-a\rangle\langle 1,-b\rangle\langle 1,-c\rangle:K^8\rightarrow K$ and is denoted by $\langle\langle a,b,c\rangle\rangle$.}
over $K$;
this ideal consists precisely of the classes $[q]$ of quadratic forms $q$ having even dimension, trivial discriminant, and trivial Hasse--Witt invariant.

The Arason invariant $e_3$ is a cohomological invariant\footnote{The expression $(a) \cup (b) \cup (c)$ is the cup product of elements in $H^1(K,\Z/2\Z)$.}
\begin{align*}
	e_3 \colon \{\text{3-Pfister forms over }K\} &\to H^3(K,\Z/2\Z) \colon \\
	\llangle a,b,c \rrangle &\mapsto (a)\cup(b)\cup(c);
\end{align*}
it extends to a well-defined group morphism $e_3 \colon I^3K \to H^3(K,\Z/2\Z)$,
that only depends on the similarity class of the quadratic form in the Witt ring (see \cite{Arason}).
It follows from the Milnor conjecture (see \cite[7.5]{V}) that this induces an isomorphism
\[ e_3 \colon I^3K/I^4K \xrightarrow{\sim} H^3(K,\Z/2\Z) . \]

Let $q$ be a form of type $E_8$; then $q$ has even dimension, trivial discriminant, and trivial Hasse--Witt invariant, so $q\in I^3K$.
In order to determine $e_3(q)$, we first rewrite $q$ in the Witt ring (note that $N = \langle 1, -\gamma^2\rangle$):
\begin{align*}
q&=N\otimes\langle 1,s_2,s_3,s_4,s_5,s_6\rangle\\
&\simeq N\otimes(\llangle -s_2,-s_3\rrangle \perp- \llangle -s_{46},-s_{56}\rrangle)\\
&\simeq\llangle\ga^2,-s_2,-s_3\rrangle \perp - \llangle\ga^2,-s_{46},-s_{56}\rrangle.
\end{align*}
It follows that
\begin{align*}
e_3(q)&=(\ga^2)\cup(-s_2)\cup(-s_3)-(\ga^2)\cup(-s_{46})\cup(-s_{56})\\
&=d(N)\cup\bigl((-s_2)\cup(-s_3)-(-s_{46})\cup(-s_{56})\bigr)\\
&=d(N)\cup c(\langle s_2,s_3, s_{23}, -s_{46}, -s_{56}, -s_{45}\rangle),
\end{align*}
where $d$ denotes the image of the discriminant in $H^1(K,\Z/2\Z)$, and $c$ denotes the Hasse--Witt invariant in $H^2(K,\Z/2\Z)$.

This invariant determines quadratic forms of type $E_8$ up to similarity:
if $q$ and $q'$ are two quadratic forms of type $E_8$ with $e_3(q)=e_3(q') \in H^3(K,\Z/2\Z)$,
then $q\equiv q' \in I^4 K$, and from \cite[Conjecture 1]{H} for $k=1$, it follows that $q$ and $q'$ are similar. 

We conclude that the invariant $e_3(q)$ completely determines the isotopy class of the structurable algebra in Theorem~\ref{th: biquat}.
(On the other hand, notice that neither $E$ nor the biquaternion algebra $(X|_K)^+$ are invariants of the quadratic form $q$
or of the corresponding structurable algebra.)
\end{remark}

\begin{remark}
The map $\nu=-1/\ga^2 \cdot q\circ \pi$ is the conjugate norm on a structurable algebra of skew-dimension one.
When we consider only the $K$\nobreakdash-part of $X$, we obtain a nice expression for $q\circ\pi$ restricted to $X|_K$. 

Indeed, in \cite[Example 4.2]{norms} it is shown that the conjugate norm of a Jordan algebra is just its generic norm. When a Jordan algebra is obtained from a central simple associative algebra, its generic norm is equal to the reduced norm of the central simple algebra.

As $X|_K$ is the Jordan algebra arising from a biquaternion algebra, we find that
\[ q(\pi(x))=-\ga^2\nu(x)=N(\ga)\Nrd(z) \]
for $x \in X|_K$, where $x = \psi(z)$ for $z\in Q_1\otimes Q_2$.
This fact also follows from \cite[Proposition 6.7]{CD}, using the Cayley--Dickson process that we will explain in the next section.
\end{remark}

\subsection{The Cayley--Dickson process for structurable algebras}

In \cite[Section 6]{CD}, Allison and Faulkner introduce a construction of structurable algebras starting from a certain class of Jordan algebras;
this procedure is a generalization of the classical Cayley--Dickson doubling process.
In the $E_8$-case, the algebra described in Theorem~\ref{th: biquat} can also be obtained by applying this Cayley--Dickson process
to the Jordan algebra obtained from a biquaternion algebra, as we will now explain.

In order to obtain a structurable algebra, one needs to start from a Jordan algebra equipped with a Jordan norm of degree 4.
We will briefly explain the Cayley--Dickson process, and refer to~\cite{CD} for more details.
\begin{definition}\label{def:CD}
Let $J$ be a Jordan algebra over K.
A form $Q \colon J\rightarrow K$ is a {\em Jordan norm of degree 4} if the following properties hold, for all $k\in K$ and all $j,j'\in J$:
\begin{compactenum}[(i)]
\item $Q(k j)=k^4Q(j)$;
\item $Q(1)=1$;
\item $Q(U_j j')=Q(j)^2Q(j')$;
\item The trace form
\[ T \colon J\times J\to k\colon (j,j')\mapsto Q(1;j)Q(1;j')-Q(1;j,j') \]
is a bilinear non-degenerate form.
\end{compactenum}
Let $Q$ be a Jordan norm of degree 4, and consider the linear automorphism $\theta$ on $J$ given by
\[ b^\theta=-b+\tfrac{1}{2}T(b,1)1 ; \]
observe that $\theta^2 = 1$.
Let $\mu \in K$.
Then $\CD(J, Q,\mu) := J\oplus s_0 J$,
with multiplication and involution given by
\begin{align*}
	(j_1+s_0j'_1)(j_2+s_0j'_2) &= j_1j_2+\mu{(j'_1j'^{\theta}_2)}^{\theta}+s_0\bigl(j_1^\theta j'_2+(j'^\theta_1 j^\theta_2)^\theta), \\
	\overline{(j+s_0j')} &= j-s_0j'^\theta,
\end{align*}
is a structurable algebra of skew-dimension one; see \cite[Theorem 6.6]{CD}.
\end{definition}

We start with the central simple biquaternion algebra equipped with the reduced norm $\Nrd$.
\begin{lemma}
Let $Q_1\otimes_K Q_2$ be a biquaternion algebra over a field $K$.
Then the reduced norm $\Nrd$ is a Jordan norm of degree 4
of the Jordan algebra $(Q_1\otimes_K Q_2)^+$.
\end{lemma}
\begin{proof}
The central simple algebra $Q_1\otimes_K Q_2$ has degree 4, so its reduced norm has degree 4 and the trace form is bilinear nondegenerate.
Since the Jordan algebra arises from a biquaternion algebra, we have $U_j j'=j j' j$, and it follows that $\Nrd(U_j j')=\Nrd(j)^2\Nrd(j')$.
\end{proof}

In order to apply the Cayley--Dickson construction to $(Q_1\otimes Q_2)^+$, we have to determine the trace map $T$ associated to $\Nrd$ explicitly.
Since $T$ is bilinear, it suffices to compute its value for elements of the form $a \otimes b$ in $Q_1\otimes Q_2$. 

It turns out that
\[ T(a \otimes b, a' \otimes b')=\Trd(a,a') \Trd(b,b')\quad \forall a,a'\in Q_1, b, b'\in Q_2,\]
where $\Trd$ is the reduced trace.
For $a=a_1+a_2\ell+a_3m+a_4 \ell m \in Q_1$, $b=b_1+b_2n+b_3r+b_4 nr \in Q_2$, we have that $T(a \otimes b, 1 \otimes 1)=4a_1b_1$,
so \[(a \otimes b)^\theta = - a \otimes b + 2 a_1 b_1 (1 \otimes 1).\]
From now on we assume $Q_1=(-s_2,-s_3)_K$ and $Q_2= (-s_{46},-s_{56})_K$.
Using the isomorphism $\psi$ from Lemma \ref{lem: biquaternion} between $Q_1\otimes_K Q_2$ and $X|_K$, $\theta$ acts on $X|_K$ as follows:
\begin{equation}\label{eqn: theta}
	\Bigl( t_{1_1}+\sum_{i\in \I} t_{i_1}v_i \Bigr)^\theta=t_{1_1}-\sum_{i\in \I} t_{i_1}v_i\quad \text{ for } t_{i_1}\in K.
\end{equation}

\begin{theorem}\label{th: CD}
The twisted Jordan biquaternion algebra $X$ defined in Theorem~\textup{\ref{th: biquat}} is isomorphic to $\CD\bigl((Q_1\otimes Q_2)^+,\Nrd,\ga^2\bigr)$.
\end{theorem}
\begin{proof}

Applying the Cayley--Dickson process on $(Q_1\otimes_K Q_2)^+\cong X|_K$ we have $\CD\bigl((Q_1\otimes_K Q_2)^+,\Nrd,\ga^2\bigr) = X|_K\oplus s_0X|_K$,
where the multiplication and involution on $X|_K\oplus s_0X|_K$ are as in Definition~\ref{def:CD}.

We now define a $K$-vector space isomorphism $\chi$ from $X$ to $X|_K\oplus s_0X|_K$:
\begin{multline*}
	(t_{1_1}+\ga t_{1_2})v_1+\sum_{i\in I}(t_{i_1}+\ga t_{i_2})v_i \\[-2ex]
		\mapsto (t_{1_1}v_1+\sum_{i\in I}t_{i_1}v_i)+s_0(-t_{1_2}v_1+\sum_{i\in I}t_{i_2}v_i)
\end{multline*}
for all $t_{i_1}, t_{i_2}\in K$.

We first check that the involution from Theorem \ref{th: biquat} is the same as the one we get from the Cayley--Dickson process. 
\begin{align*}
\chi(\overline{t_1v_1+\sum_{i\in \I}t_iv_i})&=\chi(t_1^\si v_1+\sum_{i\in \I}t_iv_i)\\
&=(t_{1_1}v_1+\sum_{i\in \I}t_{i_1}v_i)+s_0(t_{1_2}v_1+\sum_{i\in \I}t_{i_2}v_i)\\
&=(t_{1_1}v_1+\sum_{i\in \I}t_{i_1}v_i)-s_0(-t_{1_2}v_1+\sum_{i\in \I}t_{i_2}v_i)^\theta\quad \text{by } \eqref{eqn: theta}\\
&=\overline{\chi(t_1v_1+\sum_{i\in \I}t_iv_i)}.
\end{align*}
It follows immediately from Definition~\ref{def:CD} that $\chi(x)\chi(y)=\chi(x\star y)$ for all $x,y\in X|_K$.
It requires a straightforward but lengthy calculation to verify that $\chi(x)\chi(y)=\chi(x\star y)$ for $x,y\in X$ as well.
\end{proof}

\begin{remark}
In the $E_6$-case, $X$ is a twisted quaternion algebra; it is a structurable algebra isomorphic to $\CD(Q_1^+,\Nrd,\ga^2)$.
In the $E_7$-case, $X$ is a twisted quaternion algebra over a quadratic field extension;
it is a structurable algebra isomorphic to $\CD\bigl((Q_1\otimes_K K(r))^+,\Nrd,\ga^2\bigr)$.
\end{remark}

\begin{remark}
	There is an alternative approach to Theorem~\ref{th: biquat} and Theorem~\ref{th: CD}, as follows.
	It is, in principle, possible to verify directly (with a very lengthy computation, or using Sage \cite{sage}) that
	\[ \chi\bigl(3x\pi(x)\bigr) = U_{\chi(x)} \bigl( s_0 \chi(x) \bigr) \]
	for all $x \in X$, where the $U$ in the right-hand side denotes the $U$-operator in the structurable algebra $\CD\bigl((Q_1\otimes Q_2)^+,\Nrd,\ga^2\bigr)$.
	It then follows immediately that $X$ and $\CD\bigl((Q_1\otimes Q_2)^+,\Nrd,\ga^2\bigr)$ are isotopic, and since $\chi(v_1) = 1 \otimes 1$,
	it follows that these algebras are, in fact, isomorphic.
	It is then possible to compute the multiplication of $\CD\bigl((Q_1\otimes Q_2)^+,\Nrd,\ga^2\bigr)$ explicitly using this isomorphism $\chi$.
\end{remark}
                                 
\subsection{The pseudo-quadratic spaces on $E_6$ and $E_7$}

In this last section, we assume that $\Omega = (K,L,q,1,X,\cdot,h,\theta)$ is a quadrangular algebra of type $E_6$ or $E_7$.
In this case, the rank one residues of the corresponding algebraic group (see Figures~\ref{fig:E6} and~\ref{fig:E7}) are classical,
and this fact is also visible at the level of the structurable algebras.
Indeed, in \cite{W67} it is shown that $X$ can be made into a $4$-dimensional (left) vector space over $E$ or over the quaternion algebra $D=(E/K,s_2s_3s_4)$,
respectively;
moreover, there is, in both cases, an anisotropic pseudo-quadratic form on this vector space $X$, denoted by $\hat{Q}$,
with the property that
\begin{equation}\label{eq:E67}
	x\pi(x) = \hat{Q}(x) * x
\end{equation}
for all $x\in X$.
(We have used the symbol $*$ to denote the scalar multiplication of $X$ over $E$ or $D$, respectively.)
We refer to \cite[Definition 3.6, and Theorems 5.3 and 5.4]{W67} for more details.

It follows from equation~\eqref{eq:E67} and Theorem~\ref{th: herm struct alg} that
the \FTS\ corresponding to this pseudo-quadratic form as in Theorem~\ref{th: herm struct alg}
is similar to the \FTS\ of $X$ as a quadrangular algebra of type $E_6$ or $E_7$.
From Lemma \ref{lem: similar and isotopic} it follows that the corresponding structurable algebras are isotopic.

It is also interesting to note that it is shown in \cite[Theorem 5.12]{W67} that
\[ q(\pi(x))=N(\hat{Q}(x)) \]
for all $x \in X$;
both sides of this expression are, up to a constant, equal to the conjugate norm of the structurable algebra (see Theorem~\ref{th: struct expliciet}),
where the left hand side corresponds to the structurable algebra arising from the quadrangular algebra of type $E_6$ or $E_7$,
and the right hand side corresponds to the structurable algebra arising from the pseudo-quadratic form.

%%%%%%%%%%%%%%%%%%%%%%%%%%%%%%%%%%%%%%%%%%%%%%%%%%%%%%%%%%%%%%%%%%%%%%%%%%%%%%%%%%%%%%%%%%%%%%%%%%%%%%%%%%%%%%%%%%
%                                                                                                                %
%  THE BIBLIOGRAPHY                                                                                              %
%                                                                                                                %
%%%%%%%%%%%%%%%%%%%%%%%%%%%%%%%%%%%%%%%%%%%%%%%%%%%%%%%%%%%%%%%%%%%%%%%%%%%%%%%%%%%%%%%%%%%%%%%%%%%%%%%%%%%%%%%%%%

\vspace*{1ex}
\hrule
\vspace*{2ex}

\footnotesize

\noindent
Lien Boelaert,
Department of Mathematics,
Ghent University \\
Krijgslaan 281, S22,
B-9000 Gent, Belgium \\
{\tt lboelaer@cage.UGent.be}

\vspace{2ex}

\noindent
Tom De Medts,
Department of Mathematics,
Ghent University \\
Krijgslaan 281, S22,
B-9000 Gent, Belgium \\
{\tt tdemedts@cage.UGent.be}

\end{document}